\documentclass[english,american]{article}
\usepackage[T1]{fontenc}
\usepackage{textcomp}
\usepackage[utf8]{inputenc}
\usepackage{geometry}
\usepackage{parskip}
\usepackage{babel}
\usepackage{enumitem}
\usepackage{amsmath}
\usepackage{amsthm}
\usepackage{amssymb}
\usepackage[all]{xy}
\usepackage[bookmarks=true,bookmarksnumbered=false,bookmarksopen=false,
 breaklinks=false,pdfborder={0 0 0},pdfborderstyle={},backref=false,colorlinks=false]
 {hyperref}
\hypersetup{
 pdfauthor={M. Praderio}}

\makeatletter

\newcommand{\mathcircumflex}[0]{\mbox{\^{}}}

\theoremstyle{plain}
\newtheorem{thm}{\protect\theoremname}[section]
\theoremstyle{plain}
\newtheorem{prop}[thm]{\protect\propositionname}
\theoremstyle{plain}
\newtheorem*{conjecture*}{\protect\conjecturename}
\theoremstyle{definition}
\newtheorem*{defn*}{\protect\definitionname}
\newcommand\thmsname{\protect\theoremname}
\newcommand\nm@thmtype{theorem}
\theoremstyle{plain}

\newenvironment{namedthm}[1][Undefined Theorem Name]{
  \ifx{#1}{Undefined Theorem Name}\renewcommand\nm@thmtype{theorem*}
  \else\renewcommand\thmsname{#1}\renewcommand\nm@thmtype{namedtheorem}
  \fi
  \begin{\nm@thmtype}}
  {\end{\nm@thmtype}}
\theoremstyle{definition}
\newtheorem{defn}[thm]{\protect\definitionname}
\theoremstyle{definition}
\newtheorem{example}[thm]{\protect\examplename}
\theoremstyle{plain}
\newtheorem{lem}[thm]{\protect\lemmaname}
\theoremstyle{remark}
\newtheorem{notation}[thm]{\protect\notationname}
\theoremstyle{remark}
\newtheorem{rem}[thm]{\protect\remarkname}
\theoremstyle{plain}
\newtheorem{cor}[thm]{\protect\corollaryname}
\theoremstyle{remark}
\newtheorem*{rem*}{\protect\remarkname}

\date{}

\usepackage[all]{xy}

\newcommand{\bjarrow}{
  \stackrel{\cong}{\to}
}

\usepackage{mathtools}
\newcommand{\lui}[2]{\prescript{#1}{}{#2}}

\newcommand{\Aut}{\operatorname{Aut}}
\newcommand{\coker}{\operatorname{coker}}
\newcommand{\Ext}{\operatorname{Ext}}
\newcommand{\GL}{\operatorname{GL}}
\newcommand{\Graph}{\operatorname{Graph}}
\newcommand{\Hom}{\operatorname{Hom}}
\newcommand{\Id}{\operatorname{Id}}
\newcommand{\img}{\operatorname{img}}
\newcommand{\limn}[1][n]{\operatorname{lim}^{#1}}
\newcommand{\limi}{\limn[i]}
\newcommand{\Mod}{\operatorname{-Mod}}
\renewcommand{\mod}{\operatorname{-mod}}
\newcommand{\Nat}{\operatorname{Nat}}
\newcommand{\Ob}{\operatorname{Ob}}
\newcommand{\Out}{\operatorname{Out}}
\newcommand{\Rep}{\operatorname{Rep}}
\newcommand{\Syl}{\operatorname{Syl}}
\newcommand{\SL}{\operatorname{SL}}
\newcommand{\Tot}{\operatorname{Tot}}

\newcommand{\R}{\mathcal{R}}
\newcommand{\Fp}[1][p]{\mathbb{F}_{#1}}

\newcommand{\F}[1][F]{\mathcal{#1}}
\renewcommand{\H}{\F[H]}
\newcommand{\G}{\F[G]}
\newcommand{\D}{\F[D]}
\newcommand{\Fc}[1][F]{\F[#1]^{\operatorname{c}}}
\newcommand{\Fcr}[1][F]{\F[#1]^{\operatorname{cr}}}

\newcommand{\C}{\mathcal{C}}

\newcommand{\OF}[1][\F]{\mathcal{O}\left(#1\right)}
\newcommand{\OFD}[2]{\mathcal{O}_{#2}\left(#1\right)}
\newcommand{\OFC}[1][\F]{\OFD{#1}{\C}}
\newcommand{\OFc}[1][\F]{\OF[#1^{\operatorname{c}}]}

\newcommand{\FAB}[2]{\F_{#1}\left(#2\right)}
\newcommand{\FSG}[1][G]{\FAB{S}{#1}}

\newcommand{\T}{\mathcal{T}}

\newcommand{\CGp}[1][G]{C_{#1}^{p',p}}
\newcommand{\CGT}[2]{C_{#1}^{#2}}
\newcommand{\CGR}[1][G]{\CGT{#1}{\R}}
\newcommand{\CFp}[1][\Lambda]{\CGp[#1]}
\newcommand{\CFR}[1][\Lambda]{\CGR[#1]}

\makeatother

\usepackage[style=numeric,doi=false,isbn=false,url=false,eprint=false, date=year, giveninits=true, maxbibnames=9]{biblatex}
\addto\captionsamerican{\renewcommand{\conjecturename}{Conjecture}}
\addto\captionsamerican{\renewcommand{\corollaryname}{Corollary}}
\addto\captionsamerican{\renewcommand{\definitionname}{Definition}}
\addto\captionsamerican{\renewcommand{\examplename}{Example}}
\addto\captionsamerican{\renewcommand{\lemmaname}{Lemma}}
\addto\captionsamerican{\renewcommand{\notationname}{Notation}}
\addto\captionsamerican{\renewcommand{\propositionname}{Proposition}}
\addto\captionsamerican{\renewcommand{\remarkname}{Remark}}
\addto\captionsamerican{\renewcommand{\theoremname}{Theorem}}
\addto\captionsenglish{\renewcommand{\conjecturename}{Conjecture}}
\addto\captionsenglish{\renewcommand{\corollaryname}{Corollary}}
\addto\captionsenglish{\renewcommand{\definitionname}{Definition}}
\addto\captionsenglish{\renewcommand{\examplename}{Example}}
\addto\captionsenglish{\renewcommand{\lemmaname}{Lemma}}
\addto\captionsenglish{\renewcommand{\notationname}{Notation}}
\addto\captionsenglish{\renewcommand{\propositionname}{Proposition}}
\addto\captionsenglish{\renewcommand{\remarkname}{Remark}}
\addto\captionsenglish{\renewcommand{\theoremname}{Theorem}}
\providecommand{\conjecturename}{Conjecture}
\providecommand{\corollaryname}{Corollary}
\providecommand{\definitionname}{Definition}
\providecommand{\examplename}{Example}
\providecommand{\lemmaname}{Lemma}
\providecommand{\notationname}{Notation}
\providecommand{\propositionname}{Proposition}
\providecommand{\remarkname}{Remark}
\providecommand{\theoremname}{Theorem}

\begin{document}
\title{An inductive approach to the Diaz-Park sharpness conjecture}
\author{Praderio Bova, Marco (TU Dresden)}
\maketitle
\begin{abstract}
We develop tools which use common fusion systems building techniques
in order to compute higher limits over the centric orbit category.
We apply these tools in order to study both the Diaz-Park sharpness
conjecture (see \cite{DiazPark15}) as well as the weaker cohomological
sharpness conjecture which predicts vanishing of higher limits only
for the cohomology Mackey functors . Our approach leads to proving
cohomological sharpness (but not sharpness) for all saturated fusion
systems over $p$-groups of either maximal nihlpotency or of rank
$2$ and all polynomial, Henke-Shpectorov and van Beek fusion systems.
This list includes all but $2$ of the cases for which cohomological
sharpness was previously known as well as most currently known families
of exotic fusion systems. For the polynomial, Henke-Shpectorov and
$6$ of the van Beek fusion systems, sharpness is also approximated
by proving vanishing of all but the first higher limits of any Mackey
functor. The distinction our approach makes between sharpness and
cohomological sharpness is somewhat surprising and interesting by
itself. Our approach draws a new connection between cohomological
sharpness and fusion system building techniques. We believe that this
connection will lead to a better understanding of both fusion systems
and Mackey functors over them.
\end{abstract}

\section{Introduction}\label{sec:Introduction}

Saturated fusion systems are a type of category that was first introduced
by Puig (see \cite{Puig06}) as a common framework to study fusion
of both $p$-subgroups and blocks of a finite group. Although originally
intended as a tool for modular representation theory, fusion systems
have found multiple applications in various other areas of both algebra
and topology (see \cite{AKO11} for more detail).

Let $G$ be a finite group, let $p\mid\left|G\right|$ be a prime
and let $S\in\Syl_{p}\left(G\right)$. The most common example of
saturated fusion system is the category $\FSG$ with objects all subgroups
of $S$ and morphisms given by conjugation via an element in $G$.
The saturated fusion systems obtained in this manner are called realizable.
Those that cannot be ``realized'' in such a manner are called exotic.
The search and classification of exotic fusion systems has been a
subject of much research (see for example \cite{AschCher10,DiazRuizViruel07,GraPar25,GPSB26,HenkeShpectorovUnpublished,LeviOliver02,Oliver14,OliverRuiz21,VanBeek25}).
However, despite some partial results (see \cite{DiazRuizViruel07,GraPar25,Oliver14}),
we are still far from obtaining a classification of (simple) saturated
fusion systems. One of the main challenges faced when looking for
exotic fusion systems is proving that the obtained system is indeed
saturated (see Definition \ref{def:saturated-fs}). Among the many
methods employed to prove saturation, two have proven to be particularly
efficient. We refer to them as the amalgamation method and the pruning
method.

\hspace*{0.5cm}The amalgamation method. Although not every saturated
fusion system $\F$ is realizable, it is always possible to build
an iterated amalgam of finite groups $G$ containing a unique (up
to $G$-conjugacy) maximal $p$-subgroup $S$ and satisfying $\F=\FSG$
(see \cite{Robinson07}). In \cite[Lemma 4.2]{BLO06}, Broto, Levi
and Oliver give sufficient conditions for fusion systems realized
by an iterated amalgam of finite groups to be saturated. The following
result, due to Clelland and Parker, is inspired by their work.
\begin{prop}[{\cite[Theorem 3.2]{ClellandParker10}}]
\label{prop:amalgamation-method}Let $p$ be a prime, let $G_{\boldsymbol{e}}\le G_{\boldsymbol{1}},G_{\boldsymbol{2}}$
be groups admitting a unique (up to $G_{x}$-conjugacy) maximal $p$-subgroup
(denoted $S_{x}$), define the amalgam $G:=G_{\boldsymbol{1}}*_{G_{\boldsymbol{e}}}G_{\boldsymbol{2}}$,
let $\Gamma$ be its associated tree (with right $G$-action) and,
for every $P\le G$, let $\Gamma^{P}\subseteq\Gamma$ be the subtree
of points fixed by $P$. Assume that:
\begin{enumerate}
\item $S_{\boldsymbol{e}}=S_{\boldsymbol{2}}$ and $\F_{\boldsymbol{1}}:=\F_{S_{\boldsymbol{1}}}\left(G_{\boldsymbol{1}}\right)$,
$\F_{\boldsymbol{2}}:=\F_{S_{\boldsymbol{2}}}\left(G_{\boldsymbol{2}}\right)$
and $\F_{\boldsymbol{e}}:=\F_{S_{\boldsymbol{2}}}\left(G_{\boldsymbol{e}}\right)$
are all saturated,
\item $\Fcr_{x}\subseteq\Fc$ for every $x\in\left\{ \boldsymbol{1},\boldsymbol{2},\boldsymbol{e}\right\} $
(see Definition \ref{def:fully-normalized-centric-and-centric-radical-essential}).
Here we define $\F:=\F_{S_{\boldsymbol{1}}}\left(G\right)$, and
\item the quotient graph $\Gamma^{P}/C_{G}\left(P\right)$ is a tree for
every $P\in\Fc$.
\end{enumerate}
Then $\F$ is saturated.
\end{prop}

This method is used for instance to prove saturation of the Clelland-Parker
(see \cite{ClellandParker10}), the Parker-Stroth (see \cite{ParkerStroth15})
and some of the other polynomial (see \cite{GPSB26}) fusion systems.

\vspace{.5\baselineskip}

\hspace*{0.5cm}The pruning method. Let $\F$ be a saturated fusion
system over a finite $p$-group $S$. The Alperin-Goldschmidt fusion
theorem tells us that there exists a family $\mathcal{E}\left(\F\right)$
of $\F$-essential subgroups of $S$ (see Definition \ref{def:fully-normalized-centric-and-centric-radical-essential})
such that $\F$ is the smallest fusion system over $S$ containing
the $\F$-automorphisms of every $P\in\mathcal{E}\left(\F\right)\cup\left\{ S\right\} $.
In \cite{ParkerSemeraro21}, Parker and Semeraro prove that, under
certain minimality conditions, some $\F$-essentials can be ``pruned''
in order to obtain a saturated fusion subsystem of $\F$. More precisely,
for every $P\le S$, they define $H_{\F}\left(P\right)\le\Aut_{\F}\left(P\right)$
as the subgroup generated by all those $\F$-automorphism that can
be extended to isomorphisms between strictly larger subgroups of $S$.
With this notation they prove the following.
\begin{prop}[{\cite[Lemmas 6.4 and 6.5]{ParkerSemeraro21}}]
\label{prop:pruning-method}Let $\H\subseteq\F$ be fusion systems
over $S$ and let $\mathcal{A}$ be a family of subgroups of $S$
such that $\F=\left\langle \H,\Aut_{\F}\left(P\right)|P\in\mathcal{A}\right\rangle _{S}$.
That is, $\F$ is the smallest fusion system over $S$ containing
$\H$ and all the listed automorphims. Assume that every $P\in\mathcal{A}$
satisfies either:
\begin{itemize}
\item $H_{\H}\left(P\right)=\Aut_{\H}\left(P\right)$ and $P\cong p_{+}^{1+2}$
is extraspecial of order $p^{3}$ and exponent $p$, or
\item for every $Q\lneq P$ then $C_{S}\left(Q\right)\not\le Q$.
\end{itemize}
If $\F$ is saturated, then $\H$ is also saturated.
\end{prop}

This method is used to prove saturation of some van Beek (see \cite{VanBeek25})
and polynomial (see \cite{GPSB26}) fusion systems. It can also be
used to obtain saturated fusion systems over $p$-groups of maximal
nihlpotency class from others (see Proposition \ref{prop:classification-essentials-maximal-class}).

\medskip{}

Another main obstacle in the search for new exotic fusion systems
is proving that the obtained fusion system is not realizable. At the
time of writing, the only known method to distinguish between realizable
and exotic fusion systems is via the classification of finite simple
groups (see \cite[Section III.7]{AKO11}). A potential method to remedy
this would be to find a counterexample to the Diaz-Park sharpness
conjecture.

\hspace*{0.5cm}Sharpness conjecture. It is known (see \cite{BLO03,Cher13}),
that any saturated fusion system $\F$ admits a unique ``$p$-completed
classifying space'' $B\F$. If $\F=\FSG$ is realizable, then $B\F$
coincides with the $p$-completion $BG_{p}^{\mathcircumflex}$ of
the classifying space of $G$ (see \cite[Proposition 1.1]{BLO030}).
Furthermore, $B\F$ can be described as the $p$-completion of a certain
homotopy colimit (see \cite[Proposition 2.2]{BLO03}). This leads
to the existence of a Bousfield-Kan spectral sequence of the form
\begin{equation}
\limi_{\OFc}\left(H^{j}\left(-;\Fp\right)\right)\Rightarrow H^{i+j}\left(B\F;\Fp\right),\label{eq:cohomology-spectral-sequence}
\end{equation}
where $\Fp$ denotes the field of $p$-elements and $\OFc$ is the
centric orbit category of $\F$ (see Definition \ref{def:orbit-category}).
In \cite[Theorem B]{DiazPark15}, Diaz and Park prove that this spectral
sequence is sharp for realizable fusion systems. In other words $\limi_{\OFc[\FSG]}\left(H^{j}\left(-;\Fp\right)\right)=0$
for every $i\ge1$ and every $j\ge0$. This sharpness provides a proof
of the well known stable elements theorem (i.e. $\lim_{\OFc[\FSG]}\left(H^{n}\left(-;\Fp\right)\right)\cong H^{n}\left(B\FSG;\Fp\right)$).
We know from \cite[Theorem 5.8]{BLO03}, that an analogous stable
elements theorem holds also for exotic fusion system. It is therefore
natural to ask whether the spectral sequence of Equation (\ref{eq:cohomology-spectral-sequence})
is sharp also for them. At the time of writing, this question remains
open. In \cite{DiazPark15}, Diaz and Park introduce the concept of
Mackey functor over a fusion systems and conjecture that the following
stronger result actually holds.
\begin{conjecture*}[Sharpness for fusion systems]
Let $S$ be a finite $p$-group, let $\F$ be a saturated fusion
system over $S$ and let $M=\left(M_{*},M^{*}\right)$ be a Mackey
functor over $\F$ with coefficients in $\Fp$ (see \cite[Definition 2.1]{DiazPark15}).
Then $\limn_{\OFc}\left(M^{*}\downarrow_{\OFc}^{\OF}\right)=0$ for
every $n\ge1$.
\end{conjecture*}
\begin{defn*}
We say that a fusion system $\F$ satisfies sharpness if the above
conjecture holds for $\F$. If the spectral sequence of Equation (\ref{eq:cohomology-spectral-sequence})
is sharp for $\F$ and every $j\ge0$ we say that $\F$ satisfies
cohomological sharpness. 
\end{defn*}
Since any cohomology functor is the contravariant part of a Mackey
functor, then any saturated fusion system satisfying sharpness also
satisfies cohomological sharpness. In \cite[Theorem B]{DiazPark15},
Diaz and Park prove that realizable fusion systems satisfy sharpness.
It follows that a counterexample to the sharpness conjecture would
provide a method (independent of the classification) to distinguish
between realizable and exotic fusion systems. On the other hand, a
positive answer to the sharpness conjecture would provide an easy
proof of the stable elements theorem and deepen our understanding
of exotic fusion systems. Therefore any answer to this conjecture
would be of interest. This has prompted much research on the subject
(see \cites{CarrDiaz25}{DiazPark15}{GlaLyn25}{GraMar23}{HLL23}{Pra26}{Yal22}).
To the author's best knowledge, all known results (other than \cite[Theorem C]{Pra26})
which prove cohomological sharpness for at least some fusion systems,
rely on clever applications of \cite[Corollary 5.16]{JackowskiMcClure92}.
During this document we deviate from this trend by relying on the
amalgamation and pruning methods in order to provide an inductive
approach to the sharpness conjecture. By doing so, we hope not only
to get closer to proving it, but also to gain a deeper understanding
of the nature of saturated fusion systems.

\hspace*{0.5cm}Our results. Assume that the conditions of Proposition
\ref{prop:amalgamation-method} are satisfied and that $G_{\boldsymbol{1}}$,
$G_{\boldsymbol{2}}$ and $G_{\boldsymbol{e}}$ are all finite. The
author proved in \cite[Theorem B]{Pra26} that, under these conditions,
$\F$ satisfies cohomological sharpness. This result relies on the
vanishing of the functor $\CGp:=H_{1}\left(\Gamma^{P}/C_{G}\left(P\right);\Fp\right)$
and, in our opinion, has two major limitations. First; since amalgamated
products are (in general) non finite, then \cite[Theorem B]{Pra26}
cannot be iteratively applied to study fusion systems arising via
the amalgamation method. Second; if $\CGp\not=0$, then the tools
provided in \cite{Pra26} are insufficient to prove cohomological
sharpness. In this document we aim to fix these limitations.

Let $S'\le S$ be finite $p$-groups, let $\mathcal{C}$ be a family
of subgroups of $S$, let $\F_{\boldsymbol{1}}$ be a fusion system
over $S$ and let $\F_{\boldsymbol{2}}$ and $\F_{\boldsymbol{e}}$
be fusion systems over $S'$ such that $\F_{\boldsymbol{e}}\le\F_{\boldsymbol{1}},\F_{\boldsymbol{2}}$.
Define the fusion system $\F:=\left\langle \F_{\boldsymbol{1}},\F_{\boldsymbol{2}}\right\rangle _{S}$,
the set $\Lambda:=\left\{ \F_{\boldsymbol{1}},\F_{\boldsymbol{2}},\F_{\boldsymbol{e}}\right\} $,
for every $\H\in\Lambda$ define the functor $\underline{\Fp}_{\OFC[\H]}^{\OFC}$
as the induction from $\OFC[\mathcal{H}]$ to $\OFC$ (see Definition
\ref{def:orbit-category}) of the constant functor $\underline{\Fp}_{\OFC[\H]}:\OFC[\H]^{\operatorname{op}}\to\Fp\mod{}$
(see Definition \ref{def:const-functor}), define the functors $CX_{0},CX_{1}:\OFC^{\operatorname{op}}\to\Fp\mod{}$
as
\begin{align}
CX_{0}: & =\underline{\R}_{\OFC[\F_{\boldsymbol{1}}]}^{\OFC}\oplus\underline{\R}_{\OFC[\F_{\boldsymbol{2}}]}^{\OFC}, & CX_{1} & :=\underline{\R}_{\OFC[\F_{\boldsymbol{e}}]}^{\OFC},\label{eq:intro-def-CX0-CX1}
\end{align}
and define the natural transformation $f:CX_{1}\to CX_{0}$ by setting
\begin{equation}
f\left(\lambda\otimes_{\OFC[\F_{\boldsymbol{e}}]}\overline{\varphi}\right):=\left(-\lambda\otimes_{\OFC[\F_{\boldsymbol{1}}]}\overline{\varphi},\lambda\otimes_{\OFC[\F_{\boldsymbol{2}}]}\overline{\varphi}\right).\label{eq:intro-def-f}
\end{equation}

It follows from \cite[Definition 2.19]{Pra26} and Lemma \ref{lem:definition-f-ker-f},
that, whenever $\F_{\boldsymbol{1}}$, $\F_{\boldsymbol{2}}$ and
$\F_{\boldsymbol{e}}$ are all realizable, the functor $\CGp$ appearing
in \cite{Pra26} coincides with the functor $\CFp:=\ker\left(f\right)$.
This leads to the following generalization of \cite[Theorem A]{Pra26}.

\phantomsection
\label{thm:A}
\begin{namedthm}[Theorem A]
Let $\mathcal{C}$ be a family of $\F$-centric subgroups of $S$
closed under $\F$-conjugacy and taking overgroups, let $M^{*}$ be
the contravariant part of a Mackey functor over $\F$ with coefficients
in $\Fp$ and write $M:=M^{*}\downarrow_{\OFc}^{\OF}$. Assume that
$\limn_{\OFc[\H]}\left(M^{*}\downarrow_{\OFc[\H]}^{\OF}\right)$ for
every $n\ge1$ and every $\H\in\Lambda$. Assume moreover that $\Fcr[E]\subseteq\mathcal{C}$
(see Definition \ref{def:fully-normalized-centric-and-centric-radical-essential})
for every $\mathcal{E}\in\Lambda\cup\left\{ \F\right\} $ . Then
\begin{enumerate}
\item \label{enu:thm-a-iso}For every $n\ge1$ there exists an isomorphism
\[
\Ext_{\OFC}^{n}\left(\CFp,M\downarrow_{\OFC}^{\OFc}\right)\cong\limn[n+2]_{\OFc}\left(M\right),
\]
where, $\Ext_{\OFC}^{n}$ is the $n^{\text{th}}$ right derived functor
of $\Nat_{\OFC}\left(-,M\right):\OFC^{\operatorname{op}}\to\Fp\mod{}$.
\item \label{enu:thm-A-exact-seq}Let $\varUpsilon:\coker\left(f^{*}\right)\to\underset{\OFC}{\Nat}\left(\CFp,M\downarrow_{\OFC}^{\OFc}\right)$
be the morphism arising from the short exact sequence $0\to\CFp\to CX_{1}\to CX_{0}$
via the functor $\Nat_{\OFC}\left(-,M\downarrow_{\OFC}^{\OFc}\right)$.
There exists an exact sequence of the form
\[
0\to\underset{\OFc}{\limn[1]}\left(M\right)\to\coker\left(f^{*}\right)\stackrel{\varUpsilon}{\to}\underset{\OFC}{\Nat}\left(\CFp,M\downarrow_{\OFC}^{\OFc}\right)\to\underset{\OFc}{\limn[2]}\left(M\right)\to0.
\]
\end{enumerate}
\end{namedthm}
As an application of Theorem \hyperref[thm:A]{A}, we prove in Proposition
\ref{prop:sharpness-2-essentials} that cohomological sharpness holds
for any saturated fusion systems $\F$ having at most $2$ $\F$-essentials
of which one is weakly $\F$ closed while the other is minimal in
$\Fcr$. It turns out that, in situations like those arising from
the pruning method, the arguments of Proposition \ref{prop:sharpness-2-essentials}
can be applied inductively. More precisely, with notation as in Proposition
\ref{prop:pruning-method}, let $P\in\mathcal{A}$ be fully $\F$-normalized
and define $\F_{\boldsymbol{1}}:=\H$, $\F_{\boldsymbol{2}}:=N_{\F}\left(P\right)$
and $\F_{\boldsymbol{e}}:=N_{\H}\left(P\right)$ (see Definition \ref{def:normalizer-fs}).
In Corollary \ref{cor:whole-vanishing-of-CFR}, we prove that $\CFp=0$.
It follows from Theorem \hyperref[thm:A]{A} and \cite[Lemma 3.6]{Pra26}
that $\G:=\left\langle \F_{\boldsymbol{1}},\F_{\boldsymbol{2}}\right\rangle _{S}$
satisfies cohomological sharpness. By redefining $\H=\G$ and $\mathcal{A}=\mathcal{A}-\left\{ P\right\} $
we can repeat the process inductively in order to obtain our second
main result.

\phantomsection
\label{thm:B}
\begin{namedthm}[Theorem B]
Let $\H\subseteq\F$ be saturated fusion systems over $S$ and let
$\mathcal{A}\subseteq\Fc$ be a family of groups such that $\F=\left\langle \H,\Aut_{\F}\left(P\right)\,|\,P\in\mathcal{S}\right\rangle $.
Assume that every $P\in\mathcal{A}$ satisfies either
\begin{itemize}
\item $H_{\H}\left(P\right)=\Aut_{\H}\left(P\right)$ and $P\cong p_{+}^{1+2}$
is extraspecial of order $p^{3}$ and exponent $p$ and , or
\item for every $Q\lneq P$ then $C_{S}\left(Q\right)\not\le Q$.
\end{itemize}
If $\H$ satisfies cohomological sharpness then so does $\F$.
\end{namedthm}
We consider the parallelism between Proposition \ref{prop:pruning-method}
and Theorem \hyperref[thm:B]{B} to be worthy of remark. It is also
worth remarking that, unlike Proposition \ref{prop:sharpness-2-essentials},
Theorem \hyperref[thm:B]{B} cannot be proved using \cite[Theorem A]{Pra26}
but the stronger Theorem \hyperref[thm:A]{A} is in fact needed.

In each inductive step of the proof of Theorem \hyperref[thm:B]{B}
we use that $\CFp=0$. However, in cases as the one studied in Subsection
\ref{subsec:last-van-beek-fusion-system}, this condition is not satisfied.
In this situation, we can still use the description of the natural
transformation $\varUpsilon$ in Theorem \hyperref[thm:A]{A} in order
to prove cohomological sharpness.

Let $Q\le S'$ be a fully $\F$-normalized element in $\mathcal{C}$.
For every $x\in\left\{ \boldsymbol{1},\boldsymbol{2},\boldsymbol{e}\right\} $
define the fusion system $\G_{x}:=N_{\F_{x}}\left(Q\right)$, the
set $\Xi:=\left\{ \G_{\boldsymbol{1}},\G_{\boldsymbol{2}},\G_{\boldsymbol{e}}\right\} $
and the fusion system $\G:=\left\langle \G_{\boldsymbol{1}},\G_{\boldsymbol{2}}\right\rangle _{N_{S}\left(Q\right)}$.
Proceeding as in Equations (\ref{eq:intro-def-CX0-CX1}) and (\ref{eq:intro-def-f})
we can also define the functors
\begin{align}
CY_{0} & :=\underline{\Fp}_{\OFC[\G_{\boldsymbol{1}}]}^{\OFC[\G]}\oplus\underline{\Fp}_{\OFC[\G_{\boldsymbol{2}}]}^{\OFC[\G]}, & CY_{1} & :=\underline{\Fp}_{\OFC[\G_{\boldsymbol{e}}]}^{\OFC[\G]},\label{eq:CY0-CY1}
\end{align}
the natural transformation $g:CY_{1}\to CY_{0}$ and $\CFp[\Xi]:=\ker\left(g\right)$.
By applying the induction functor $\uparrow_{\OFC[\G]}^{\OFC}$ to
$\CFp[\Xi]$, $CY_{0}$ and $CY_{1}$ we obtain the following commutative
diagram
\[
\xymatrix{\CFp[\Xi]\uparrow_{\OFC[\G]}^{\OFC}\ar@{->}[r]^{\phantom{..}\iota_{Y}\uparrow_{\OFC[\G]}^{\OFC}} & CY_{1}\uparrow_{\OFC[\G]}^{\OFC}\ar@{->}[r]^{\phantom{.......}g\uparrow_{\OFC[\G]}^{\OFC}}\ar@{->>}[d]^{\pi_{1}} & CY_{0}\ar@{->>}[d]^{\pi_{0}}\\
\CFp\ar@{->}[r]^{\iota_{X}} & CX_{1}\ar@{->}[r]^{f} & CX_{0}
}
,
\]
where $\iota_{X}$ and $\iota_{Y}$ denote the natural inclusions
and $\pi_{1}$ and $\pi_{0}$ denote the natural projections arising
from the inclusions $\G_{x}\subseteq\F_{x}$. It follows from the
above diagram that $\pi_{1}\iota_{Y}\uparrow_{\OFC[\G]}^{\OFC}$ has
image contained in $\CFp$. This leads to the canonical natural transformation
\begin{equation}
\Gamma:\underset{x\otimes\overline{\varphi}}{\CFp[\Xi]\uparrow_{\OFC[N_{\F}\left(Q\right)]}^{\OFC}}\underset{\rightarrow\vphantom{x\otimes\overline{\varphi}}}{\longrightarrow\vphantom{\CFp[\Xi]\uparrow_{\OFC[N_{\F}\left(Q\right)]}^{\OFC}}}\underset{x\overline{\varphi}}{\CFp\vphantom{\uparrow_{\OFC[N_{\F}\left(Q\right)]}^{\OFC}}}.\label{eq:def-gamma}
\end{equation}
On the other hand, for any functor $M:\OFC^{\operatorname{op}}\to\Fp\mod{}$,
the epimorphism $\pi_{1}$ leads to a monomorphism $\pi_{1}^{*}:\Nat_{\OFC}\left(CX_{1},M\downarrow_{\OFC}^{\OFc}\right)\hookrightarrow\Nat_{\OFC}\left(CY_{1}\uparrow_{\OFC[\G]}^{\OFC},M\downarrow_{\OFC}^{\OFc}\right)$.
Our third main result is the following.

\phantomsection
\label{thm:C}
\begin{namedthm}[Theorem C]
Adopt the notation of Theorem \hyperref[thm:A]{A} and let $Q$,
$\G$, $\Gamma$ and $\pi_{1}^{*}$ be as above. Assume that $\mathcal{E}$
is saturated and that $\Fcr[E]\subseteq\mathcal{C}$ for every $\mathcal{E}\in\Lambda\cup\left\{ \F\right\} $
and that $\limn_{\OFc[\H]}\left(M^{*}\downarrow_{\OFc[\H]}^{\OF}\right)$
for every $n\ge1$ and every $\H\in\Lambda$. Assume moreover that
$N_{\F}\left(Q\right)=\G$, that $\Gamma$ is an isomorphism and that
$\pi_{1}^{*}$ is a split monomorphism. Then $\limn_{\OFc}\left(M\right)=0$
for every $n\ge2$.

In particular, if each $\F_{x}$ satisfies cohomological sharpness,
then so does $\F$.
\end{namedthm}
In Corollary \ref{cor:relation-induction-no-induction} we provide
sufficient conditions for the morphism $\Gamma$ in the above to be
an isomorphism. On the other hand, the condition $N_{\F}\left(Q\right)=\G$
is satisfied if $Q$ is weakly $\F$ closed (as in Subsection \ref{subsec:last-van-beek-fusion-system})
or in fusion systems like the Benson-Solomon fusion systems (see \cite[Lemma 7.10(a)]{AschCher10}).
In both of these cases we also have that $Q\trianglelefteq S$. Finally,
in Proposition \ref{prop:sufficient-condition-for-splitting.} we
prove that $\pi_{1}^{*}$ is a split monomorphism whenever $M^{*}=H^{n}\left(-;\Fp\right)$
for some $n\ge0$, the fusion system $\F_{\boldsymbol{e}}$ is realizable
and $Q\trianglelefteq S'$. Once again these conditions are met in
fusion systems as the one in Subsection \ref{subsec:last-van-beek-fusion-system}
or the Benson-Solomon fusion systems. It is worth noticing that the
proof of Proposition \ref{prop:sufficient-condition-for-splitting.}
relies heavily on the stable elements theorem (see \cite[Theorem 5.8]{BLO03})
and on $H^{n}\left(-;\Fp\right)$ being both a global and a cohomological
Mackey functor. This suggests that, if sharpness were false and cohomological
sharpness true, then a counterexample should be found among simple
Mackey functors $S_{V,P}$ that either cannot be lifted to global
Mackey functors or, if they do, satisfy $\lim_{\OFc[\FSG]}\left(S_{V,P}^{*}\right)\not\cong S_{V,P}^{*}\left(G\right)$
for some realizable fusion subsystem.

Our fourth and last main result consists on a case by case application
of Theorems \hyperref[thm:B]{B} and \hyperref[thm:C]{C}.

\phantomsection
\label{thm:D}
\begin{namedthm}[Theorem D]
Let $\F$ be one of the following fusion systems:
\begin{enumerate}
\item \label{enu:theorem-D-polynomial}Any polynomial fusion system (in
the sense of \cite{GPSB26}).
\item \label{enu:theorem-D-maximal-class}Any saturated fusion system over
a $p$-group of maximal nihlpotency class.
\item \label{enu:theorem-D-rank-2}Any saturated fusion system over a $p$-group
of rank $2$ for $p$ odd.
\item \label{enu:theorem-D-henke-Shpectorov}Any Henke-Shpectorov fusion
system (as defined in \cite{HenkeShpectorovUnpublished} and then
again in \cite[Theorem 5.8]{ParkerSemeraro21} and characterized in
\cite[Theorem 1.1]{GPSB26}).
\item \label{enu:theorem-D-van-beek}Any of the $16$ van Beek fusion system
(as defined in \cite{VanBeek25}).
\end{enumerate}
Then $\F$ satisfies cohomological sharpness.
\end{namedthm}
\begin{proof}
Item (\ref{enu:theorem-D-polynomial}) is proven in Proposition \ref{prop:sharpness-polynomial},
Item (\ref{enu:theorem-D-maximal-class}) is proven in Proposition
\ref{prop:sharpness-maximal-nihlpotency-class}, Item (\ref{enu:theorem-D-rank-2})
is proven in Proposition \ref{prop:sharpness-rank-2}. Item (\ref{enu:theorem-D-henke-Shpectorov})
is proven in Proposition \ref{prop:sharpness-Henke-Shpectorov} and
Item (\ref{enu:theorem-D-van-beek}) is proven in Propositions \ref{prop:sharpness-van-beek-F3},
\ref{prop:sharpness-van-beek-E1-F3}, \ref{prop:sharpness-van-beek-M},
\ref{prop:sharpness-van-beek-E1-M} and \ref{prop:Sharpness-last-Van-Beek}.
\end{proof}
To the author's best knowledge, the only fusion systems listed in
Theorem \hyperref[thm:D]{D} for which cohomological sharpness was
already known are:
\begin{itemize}
\item Fusion systems over $p$-groups with an abelian subgroup of index
$p$ (see \cite[Theorem C]{DiazPark15}). This includes some (but
not all) polynomial fusion systems, fusion systems over some some
(but not all) groups of maximal class and fusion systems over some
(but not all) $p$-groups of rank $2$ with $p$ odd.
\item The Clelland-Parker and the Parker-Stroth fusion systems (see \cite[Theorems 1.1 and 1.5]{HLL23}
and \cite[Theorem C]{PraderioPhD24}). These are strictly contained
in the family of polynomial fusion systems.
\item Fusion systems over Sylow $p$-subgroups of $G_{2}\left(p\right)$
with $p\ge5$ (see \cite{GraMar23}). This includes some (but not
all) fusion systems over a $p$-group of maximal nihlpotency class.
\end{itemize}
The proofs of cohomological sharpness that we provide in Theorem \hyperref[thm:D]{D}
is independent from previous results with the sole exception of realizable
fusion systems and of fusion systems over a Sylow $7$ subgroup of
$G_{2}\left(7\right)$. On the other hand, at the time of writing
and to the authors best knowledge, there exist only $2$ known (families
of) fusion systems for which cohomological sharpness is known to hold
and are not covered in Theorem \hyperref[thm:D]{D}. These are:
\begin{enumerate}
\item Saturated fusion systems over a $p$-group with a maximal abelian
subgroup of index $p$ (see \cite[Theorem C]{DiazPark15}).
\item The smallest Benson-Solomon fusion system (see \cite[Theorems 1.1 and 4.1]{HLL23}).
\end{enumerate}
Attempting to prove cohomological sharpness in the first of the above
cases by using the methods developed in this document requires a study
on the centric-radical subgroups that falls outside the scope of this
document. The same statement holds for the exotic fusion systems over
the Sylow $7$-subgroups of $G_{2}\left(7\right)$ or the Oliver-Ruiz
fusion systems (see \cite{OliverRuiz21}). In the later case cohomological
sharpness is still not known. The study of cohomological sharpness
for the Benson-Solomon fusion systems (see \cite{AschCher10,LeviOliver02})
using Theorem \hyperref[thm:C]{C} is deferred to future work. 

\medskip{}

\hspace*{0.5cm}Organization of the paper: We start in Section \ref{sec:trees-of-fusion-systems}
with a quick review on fusion systems and of trees of fusion systems
(see \cite{Semeraro14}). Section \ref{sec:Higher-limits-and-amalgams}
constitutes the core of this document and, in it, we prove Theorem
\hyperref[thm:A]{A}. In Section \ref{sec:sharpness_via_pruning}
we provide some sufficient conditions for the functor $\CFp$ to vanish
(see Corollary \ref{cor:whole-vanishing-of-CFR}) and prove Theorem
\hyperref[thm:B]{B}. It this section we also prove that cohomological
sharpness holds for all but one of the fusion systems listed in Theorem
\hyperref[thm:D]{D}. For general Mackey functors over either a polynomial,
Henke-shpectorov or $6$ of the van Beek fusion systems, we also prove
vanishing of all but the first higher limits (see Propositions \ref{prop:sharpness-polynomial},
\ref{prop:sharpness-Henke-Shpectorov}, \ref{prop:sharpness-van-beek-F3},
\ref{prop:sharpness-van-beek-E1-F3} and \ref{prop:sharpness-van-beek-M}).
We conclude in Section \ref{sec:sharpness_via_normalizers} where
we prove Theorem \hyperref[thm:C]{C} and use is to prove cohomological
sharpness for the remaining fusion system listed in Theorem \hyperref[thm:D]{D}.

\hspace*{0.5cm}Acknowledgment. The author would like to thank Chris
Parker, Jason Semeraro, Martin van Beek, Valentina Grazian, Justin
Lynd, Andrew Chermak, Guille Carrión and Antonio Díaz for their comments
and insightful conversations.

During the development of this article the author was supported by
the TU Dresden university via the grants U-000902-830-C11-1010103,
F-013038-541-000-1010103 and L-000150-721-110-1010100.

\section{Fusion systems and trees of fusion systems}\label{sec:trees-of-fusion-systems}

In this section we briefly recall some of the most common notation
regarding fusion systems as well as the concept of trees of fusion
systems. We refer the interested reader to \cite[Part I]{AKO11} for
a more in depth introduction to fusion systems and to \cite{Semeraro14}
for a more through study of trees of fusion systems.
\begin{defn}
Let $p$ be a prime and let $S$ be a finite $p$-group. A fusion
system over $S$ is a category $\F$ whose objects are all the subgroups
of $S$ and whose morphisms are injective group morphisms such that:
\begin{enumerate}
\item $\F$ contains all morphisms obtained via conjugation by an element
in $S$.
\item If a morphism $\varphi$ in $\F$ is a group isomorphism then $\varphi^{-1}$
is also a morphism in $\F$.
\item For every morphism $\varphi:P\to Q$ in $\F$, the morphism $\tilde{\varphi}:P\to\varphi$$\left(P\right)$
obtained by restricting $\varphi$ to its image is also a morphism
in $\F$.
\end{enumerate}
\end{defn}

\begin{example}
\label{exa:realization-fs}Let $G\ge S$ be a group. The fusion system
on $S$ realized by $G$ is the fusion system $\FSG$ over $S$ whose
morphisms are given by conjugation by elements of $G$.
\end{example}

Some notable subgroups of $S$ related to a fusion system are the
following.
\begin{defn}
\label{def:fully-normalized-centric-and-centric-radical-essential}Let
$\F$ be a fusion system over $S$ and let $P\le S$. We say that:
\begin{enumerate}
\item $P$ is fully $\F$-normalized if for every $\F$-conjugate ($Q\cong_{\F}P$)
of $P$ the inequality $\left|N_{S}\left(P\right)\right|\ge\left|N_{S}\left(Q\right)\right|$
holds.
\item $P$ is $\F$-centric if for every $Q\cong_{\F}P$ then $C_{S}\left(Q\right)\le Q$.
We denote by $\Fc$ the family of all $\F$-centric subgroups of $S$.
\item $P$ is $\F$-radical if $\Aut_{P}\left(P\right)=O_{p}\left(\Aut_{\F}\left(P\right)\right)$,
where $\Aut_{P}\left(P\right)$ denotes the automorphisms of $P$
given by conjugation by an element of $P$.
\item $P$ is $\F$-centric-radical if it is $\F$-centric and $\F$-radical.
We denote by $\Fcr$ the family of all $\F$-centric-radical subgroups
of $S$.
\item $P$ is $\F$-essential if $\Out_{\F}\left(P\right):=\Aut_{\F}\left(P\right)/\Aut_{P}\left(P\right)$
contains a strongly $p$-embedded subgroup. That is $p\mid\left|\Out_{\F}\left(P\right)\right|$
and there exists $Q\in\Syl_{p}\left(\Out_{\F}\left(P\right)\right)$
and $Q\le H\le\Out_{\F}\left(P\right)$ such that $Q^{x}\cap H=\left\{ \Id\right\} $
for every $x\in\Out_{\F}\left(P\right)-H$. We denote by $\mathcal{E}\left(\F\right)$
the family of all $\F$-essential subgroups of $S$.
\end{enumerate}
\end{defn}

Throughout this document we often work with the orbit category and
the normalizer subsystems of a fusion system.
\begin{defn}
\label{def:orbit-category}Let $\F$ be a fusion system. The orbit
category of $\F$ is the category whose objects are the same as the
objects of $\F$ and whose morphisms are the cosets $\Hom_{\OF}\left(P,Q\right):=\Aut_{Q}\left(Q\right)\backslash\Hom_{\F}\left(P,Q\right)$.
Here the action is given by post composition. Given a family $\mathcal{C}$
of groups we define the $\mathcal{C}$-orbit category of $\F$ as
the full subcategory $\OFC$ of $\OF$ whose objects are the subgroups
of $S$ contained in $\mathcal{C}$. If $\mathcal{C}=\Fc$ we define
$\OFc:=\OFC$ and call it the centric orbit category of $\F$.
\end{defn}

\begin{defn}
\label{def:normalizer-fs}Let $P\le S$. We define the normalizer
of $\F$ on $P$ as the fusion system $N_{\F}\left(P\right)$ over
$N_{S}\left(P\right)$ with morphisms
\[
\Hom_{N_{\F}\left(P\right)}\left(A,B\right):=\left\{ \varphi\in\Hom_{\F}\left(A,B\right)\,|\,\exists\hat{\varphi}\in\Hom_{\F}\left(AP,BP\right)\text{ s.t. }\hat{\varphi}\left(P\right)=P\text{ and }\hat{\varphi}\left(x\right)=\varphi\left(x\right)\;\forall x\in A\right\} .
\]
\end{defn}

In the literature, it is not uncommon for authors to simply call fusion
system what it also known as saturated fusion system.
\begin{defn}
\label{def:saturated-fs}Let $\F$ be a fusion system over $S$. For
every $\varphi:P\to Q$ in $\F$ define
\[
N_{\varphi}:=\left\{ y\in N_{S}\left(P\right)\mid\exists z\in N_{S}\left(\varphi\left(P\right)\right)\text{ s.t. }\varphi\left(x^{y}\right)=\varphi\left(x\right)^{z}\;\forall x\in P\right\} .
\]
We say that $\F$ is saturated if the following are satisfied:
\begin{description}
\item [{Sylow$\;$axiom}] $\Aut_{S}\left(S\right)\in\Syl_{p}\left(\Aut_{\F}\left(S\right)\right)$.
\item [{Extension$\;$axiom}] For every $\varphi:P\to S$ in $\F$, if
$\varphi\left(P\right)$ is fully $\F$-normalized then there exists
$\hat{\varphi}:N_{\varphi}\to S$ such that $\varphi\left(x\right)=\hat{\varphi}\left(x\right)$
for every $x\in P$.
\end{description}
\end{defn}

The following are well known examples of saturated fusion systems.
\begin{example}
$\phantom{.}$
\begin{itemize}
\item If $G$ is finite and $S\in\Syl_{p}\left(G\right)$ then $\FSG$ is
saturated. We call saturated fusion systems of this form realizable.
A non realizable saturated fusion system is called exotic.
\item If $\F$ is saturated and $P\le S$ is fully $\F$-normalized then
$N_{\F}\left(P\right)$ is saturated.
\end{itemize}
\end{example}

Throughout this document we often use the following elementary result.
\begin{lem}
\label{lem:sharpness-for-normalizers}Let $P\in\Fc$ be fully $\F$-normalized.
If $\F$ is saturated then $N_{\F}\left(P\right)$ satisfies sharpness.
\end{lem}

\begin{proof}
We know from \cite[Proposition C]{BCGLO05} that $N_{\F}\left(P\right)$
is realizable. The result follows from \cite[Theorem B]{DiazPark15}.
\end{proof}
we also make use of the following well known result.
\begin{lem}[{\cite[Proposition 4.46]{Craven11}}]
\label{lem:centric-radical-contained}Let $P\le S$ be fully $\F$-normalized.
If $\F$ is saturated then every $N_{\F}\left(P\right)$-centric-radical
contains $P$.
\end{lem}

\medskip{}

Building on some previous work of Broto, Levi and Oliver (see \cite{BLO06}),
Semeraro developed in \cite{Semeraro14} the concept of tree of fusion
systems and used it in order to study sufficient conditions for saturation.
Throughout this document we are mostly concerned with the simplest
non-trivial case of tree of fusion system satisfying what Semeraro
calls Hypothesis (H). This simplest of cases is described by adopting
the following notation.
\begin{notation}
\label{nota:simple-tree-of-fusion-systems}$\phantom{.}$
\begin{enumerate}
\item $\T$ is the tree with only two vertexes $\left\{ \boldsymbol{1},\boldsymbol{2}\right\} $
and a single edge $\boldsymbol{e}$ between them.
\item $\T$ can be viewed as a category with objects $\Ob\left(\T\right)=\left\{ \boldsymbol{1},\boldsymbol{2},\boldsymbol{e}\right\} $
and only two non identity morphisms $\boldsymbol{e}\to\boldsymbol{1}$
and $\boldsymbol{e}\to\boldsymbol{2}$.
\item $S\ge S'$ are finite $p$ groups and $\F_{\boldsymbol{1}}$ and $\F_{\boldsymbol{2}}$
be fusion systems over $S$ and $S'$ respectively.
\item $\F=\left\langle \F_{\boldsymbol{1}},\F_{\boldsymbol{2}}\right\rangle _{S}$
is the smallest fusion system over $S$ containing both $\F_{\boldsymbol{1}}$
and $\F_{\boldsymbol{2}}$ and $\F_{\boldsymbol{e}}$ is a fusion
system over $S'$ such that $\F_{\boldsymbol{e}}\subseteq\F_{\boldsymbol{1}}\cap\F_{\boldsymbol{2}}$.
\item Define $S_{\boldsymbol{1}}:=S$ and $S_{\boldsymbol{2}}:=S_{\boldsymbol{e}}:=S'$.
\item Define $\Lambda:=\left\{ \F_{\boldsymbol{1}},\F_{\boldsymbol{2}},\F_{\boldsymbol{e}}\right\} $.
\end{enumerate}
\end{notation}

\begin{rem}
With the above notation it is possible to define a functor $\boldsymbol{\Lambda}$
(resp, $\mathcal{S}$) from $\T$ to the category of $p$ fusion system
(resp. $p$-groups) which sends $x$ to $\F_{x}$ (resp. $S_{x}$)
and the morphisms to the natural inclusions. It follows from \cite[Lemma 2.6]{Semeraro14},
that $\boldsymbol{\Lambda}$ admits a colimit $\boldsymbol{\Lambda}_{\T}$
and that $\boldsymbol{\Lambda}_{\T}=\F$.
\end{rem}

In \cite[Theorem 4.1]{Semeraro14}, Semeraro shows that if the $\F_{x}$
are saturated, $\Fcr_{x}\subseteq\Fc$ and a certain graph (which
we denote as $\Rep_{\F}\left(P,\Lambda\right)$) is a tree for every
$P\in\Fc$ then $\F$ is saturated. This graph can be defined as follows.
\begin{defn}
Let $S'\le S$ be finite $p$-groups and let $\H\subseteq\F$ be fusion
systems over $S'$ and $S$ respectively. For every $P\le S$ we define
the set
\[
\Rep_{\F}\left(P,\H\right):=\Hom_{\F}\left(P,S'\right)/\sim,
\]
where $\varphi\sim\psi$ if and only if there exists an isomorphism
$\theta:\varphi\left(P\right)\to\psi\left(P\right)$ in $\H$ such
that $\theta\left(\varphi\left(x\right)\right)=\psi\left(x\right)$
for every $x\in P$. For every $\varphi\in\Hom_{\F}\left(P,S'\right)$
we denote by $\left[\varphi\right]_{\H}$ its equivalence class in
$\Rep_{\F}\left(P,\H\right)$.
\end{defn}

Adopting Notation \ref{nota:simple-tree-of-fusion-systems}, we can
view $\Rep_{\F}\left(P,\F_{x}\right)$ as a discrete topological space.
This allows us to define the homotopy colimit of the resulting functor
$\Rep_{\F}\left(P,-\right):\T\to\operatorname{Top}$. As shown in
\cite[Lemma 2.8]{Semeraro14}, this homotopy colimit can be described
as the geometric realization of the graph $\Rep_{\F}\left(P,\Lambda\right)$.
\begin{defn}
\label{def:RepF-P-lambda}Adopt Notation \ref{nota:simple-tree-of-fusion-systems}.
For every $P\le S$, we define $\Rep_{\F}\left(P,\Lambda\right)$
as the bipartite graph whose vertex and edge sets are given by
\begin{align*}
V & :=\Rep_{\F}\left(P,\F_{\boldsymbol{1}}\right)\sqcup\Rep_{\F}\left(P,\F_{\boldsymbol{2}}\right), & E & :=\Rep_{\F}\left(P,\F_{\boldsymbol{e}}\right),
\end{align*}
where $\left[\varphi\right]_{\F_{\boldsymbol{e}}}\in E$ has incident
vertices $\left[\varphi\right]_{\F_{\boldsymbol{2}}}$ and $\left[\iota_{S'}^{S}\varphi\right]_{\F_{\boldsymbol{1}}}$.
\end{defn}

For every $\varphi:Q\to P$ in $\F$ there is a natural morphism $\Rep_{\F}\left(P,\Lambda\right)\to\Rep_{\F}\left(Q,\Lambda\right)$
sending $\left[\psi\right]_{\H}$ to $\left[\psi\varphi\right]_{\H}$
for every $\H\in\Lambda$. Moreover, if $Q=P$ and $\varphi$ is given
by conjugation by an element in $P$, then this morphism is the identity
morphism. In other words, we obtain the following result, with which
we conclude this section.
\begin{lem}
\label{lem:functor-to-graph}Adopt Notation \ref{nota:simple-tree-of-fusion-systems}.
There exists a functor $\Rep_{\F}\left(-,\Lambda\right):\OF^{\operatorname{op}}\to\Graph$
which sends every $P\le S$ to the graph $\Rep_{\F}\left(P,\Lambda\right)$.
\end{lem}

\section{Isomorphisms and exact sequences for higher limits}\label{sec:Higher-limits-and-amalgams}

Throughout this section we adopt the following notation
\begin{notation}
\label{nota:C-included-in-notation}Notation \ref{nota:simple-tree-of-fusion-systems}
is adopted and we use the symbol $\mathcal{C}$ in order to denote
a family of subgroups of $S$ containing $S'$ and closed under $\F$-conjugation
and taking overgroups.
\end{notation}

Many results in this section are slight generalizations of those presented
by the author in \cite[Section 3]{Pra26}. However, in contrast with
\cite[Section 3]{Pra26}, we take a more algebraic approach and compute
the morphism $\varUpsilon$ of Theorem \hyperref[thm:A]{A}.
\begin{defn}[{\cite[Definition 2.19]{Pra26}}]
\label{def:CFR}Let $\R$ be a commutative ring. By viewing graphs
as $CW$-complexes in the natural way, we can compose the functor
$\Rep_{\F}\left(-,\Lambda\right)$ of Lemma \ref{lem:functor-to-graph}
with the first homology functor $H_{1}\left(-;\R\right)$ to define
the functor $\CFR:\OFC^{\text{op}}\to\R\mod{}$. When $\R=\Fp$ is
the field of $p$-elements we write $\CFp:=\CGT{\Lambda}{\Fp}$.
\end{defn}

For every $\mathcal{H}\in\Lambda$ and every $P\in\mathcal{C}$, define
$\R\Rep_{\F}\left(P,\mathcal{H}\right)$ as the free $\R$-module
with basis indexed by $\Rep_{\F}\left(P,\mathcal{H}\right)$ and let
$CX_{*}$ be the $\R$ cellular chain complex of $\Rep_{\F}\left(P,\Lambda\right)$
(seen as a $CW$-complex). Since $\Rep_{\F}\left(P,\Lambda\right)$
is the geometric realization of a graph, then $CX_{n}=0$ for every
$n\not=0,1$ while
\begin{align*}
CX_{0}\left(P\right) & :=\R\Rep_{\F}\left(P;\F_{\boldsymbol{1}}\right)\oplus\R\Rep_{\F}\left(P;\F_{\boldsymbol{2}}\right), & CX_{1}\left(P\right) & :=\R\Rep_{\F}\left(P;\F_{\boldsymbol{e}}\right).
\end{align*}
Moreover, the differential $f:CX_{1}\to CX_{0}$ is given by setting
\[
f\left(\lambda\left[\varphi\right]_{\F_{\boldsymbol{e}}}\right):=\left(-\lambda\left[\iota_{S'}^{S}\varphi\right]_{\F_{\boldsymbol{1}}},\lambda\left[\varphi\right]_{\F_{\boldsymbol{2}}}\right),
\]

for every $P\in\mathcal{C}$, every $\lambda\in\R$ and every $\varphi\in\Hom_{\F}\left(P,S'\right)$.
Let us provide a more suitable interpretation of $CX_{0}$, $CX_{1}$
and $f$.
\begin{defn}
\label{def:const-functor}Let $\R$ be a commutative ring . We define
$\underline{\R}:\OFC^{\text{op}}\to\R\mod{}$ as the constant functor
sending every object to $\R$. For every $\mathcal{H}\in\Lambda$,
we further define $\underline{\R}_{\OFC[\mathcal{H}]}^{\OFC}:\OFC^{\text{op}}\to\R$
as the functor obtained by first restricting $\underline{\R}$ to
$\OFC[\mathcal{H}]$ and the inducing it back to $\OFC$.
\end{defn}

\begin{rem}
\label{rem:equivalence-of-index-cat-RC-mod}Let $\boldsymbol{C}$
be a small category. It is well known that there exists an equivalence
between the category of right $\R\boldsymbol{C}$-modules and the
category $\left(\R\Mod\right)^{\boldsymbol{C}^{\operatorname{op}}}$
of contravariant functors from $\boldsymbol{C}$ to $\R\Mod$. This
equivalence is what we use in Definition \ref{def:const-functor}
when we talk about induction and restriction. More precisely, $\underline{\R}_{\OFC[\mathcal{H}]}^{\OFC}$
is obtained by viewing $\underline{\R}$ as a right $\R\OFC[\H]$-module,
viewing $\R\OFC$ simultaneously as a left $\R\OFC[\H]$-module and
a right $\R\OFC$-module and then viewing their tensor product (i.e.
the induced module) as a functor.
\end{rem}

As shown in \cite[Lemma 4.7]{Yal22}, for every $P\in\mathcal{C}$
we have that 
\[
\underline{\R}_{\OFC[\mathcal{H}]}^{\OFC}\left(P\right)=\left\langle \left[\varphi\right]_{\H}\mid\varphi\left(P\right)\in\Ob\left(\H\right)\right\rangle \subseteq\R\Rep_{\F}\left(P,\H\right).
\]
We can therefore redefine $CX_{0}$ and $CX_{1}$ as the functors
\begin{align*}
CX_{0} & =\underline{\R}_{\OFC[\F_{\boldsymbol{1}}]}^{\OFC}\oplus\underline{\R}_{\OFC[\F_{\boldsymbol{2}}]}^{\OFC}, & CX_{1} & =\underline{\R}_{\OFC[\F_{\boldsymbol{e}}]}^{\OFC}.
\end{align*}
with this notation, whenever $\R=\Fp$, the natural transformation
$f$ coincides with the one defined by Equation (\ref{eq:intro-def-f}).
We conclude the following.
\begin{lem}
\label{lem:definition-f-ker-f}The functor $\CFR$ of Definition \ref{def:CFR}
is naturally isomorphic to the functor $\ker\left(f\right)$. In particular,
our notation is consistent with that of Theorem \hyperref[thm:A]{A}.
\end{lem}

Moreover we also have the following
\begin{lem}
\label{lem:coker-f}With notation as in Lemma \ref{lem:definition-f-ker-f}
$\coker\left(f\right)\cong\underline{\R}$.
\end{lem}

\begin{proof}
Fix $P\in\mathcal{C}$, let $\varphi:P\bjarrow Q$ be any $\F$-isomorphism
and for every $x\in CX_{0}\left(P\right)$ denote by $\overline{x}\in\coker\left(f\right)\left(P\right)$
its natural projection. Since $\F=\left\langle \F_{1},\F_{2}\right\rangle _{S}$,
then we can write $\varphi=\psi_{\infty}\theta_{n}\psi_{n}\theta_{n-1}\psi_{n-1}\cdots\psi_{1}$
where the $\psi_{i}$ and $\theta_{i}$ are isomorphisms in $\F_{\boldsymbol{1}}$
and $\F_{\boldsymbol{2}}$ respectively. For every $i=1,\dots,n$,
define $\varphi_{i}^{1}:=\iota^{S}\theta_{i}\psi_{i}\cdots\psi_{1}$
and $\varphi_{i}^{2}:=\iota^{S'}\psi_{i}\theta_{i-1}\cdots\psi_{1}$
where $\iota^{S}$ and $\iota^{S'}$ denote the natural inclusions
into $S$ and $S'$ respectively. The result follows from the identities
below
\begin{align*}
\overline{\left(\left[\iota^{S}\varphi\right]_{\F_{\boldsymbol{1}}},0\right)} & =\overline{\left(\left[\varphi_{n}^{1}\right]_{\F_{\boldsymbol{1}}},0\right)}=\overline{\left(0,-\left[\varphi_{n}^{2}\right]_{\F_{\boldsymbol{2}}}\right)}=\overline{\left(\left[\varphi_{n-1}^{1}\right]_{\F_{\boldsymbol{1}}},0\right)}=\cdots=\overline{\left(\left[\iota^{S}\right]_{\F_{\boldsymbol{1}}},0\right)}.
\end{align*}
\end{proof}
Fix a functor $M:\OFC^{\text{op}}\to\R\mod{}$, let $P_{*,*}$ be
a Cartan-Eilenberg projective resolution of the chain complex $CX_{*}$
and define de double cocomplex $Q^{*,*}:=\Nat_{\OFC}\left(P_{*,*},M\right)$.
We depict $CX_{*}$ , $P_{*,*}$ and $Q^{*,*}$in the following diagram
for future reference.
\begin{align}
\xymatrix{0\ar@{->}[d] & 0\ar@{->}[d] & 0\ar@{->}[d]\\
CX_{1}\ar@{->}[d]_{f} & P_{0,1}\ar@{->}[d]_{d_{0,1}^{v}}\ar@{->>}[l]_{\varepsilon_{1}} & P_{1,1}\ar@{->}[d]_{d_{1,1}^{v}}\ar@{->}[l]_{d_{1,1}^{h}} & \ar@{->}[l]_{d_{2,1}^{h}}\cdots\\
CX_{0}\ar@{->}[d] & P_{0,0}\ar@{->}[d]\ar@{->>}[l]_{\varepsilon_{0}} & P_{1,0}\ar@{->}[d]\ar@{->}[l]_{d_{1,0}^{h}} & \ar@{->}[l]_{d_{2,0}^{h}}\cdots\\
0 & 0 & 0
}
, &  & \xymatrix{ & 0 & 0\\
0\ar@{->}[r] & Q^{0,1}\ar@{->}[u]\ar@{->}[r]^{d_{h}^{0,1}} & Q^{1,1}\ar@{->}[u]\ar@{->}[r]^{d_{h}^{1,1}} & \cdots\\
0\ar@{->}[r] & Q^{0,0}\ar@{->}[u]^{d_{v}^{0,0}}\ar@{->}[r]^{d_{h}^{0,0}} & Q^{1,0}\ar@{->}[u]^{d_{v}^{1,0}}\ar@{->}[r]^{d_{h}^{1,0}} & \cdots\\
 & 0\ar@{->}[u] & 0\ar@{->}[u]
}
.\label{eq:def-P-and-Q}
\end{align}

\begin{rem}
\label{rem:switched-indexes}The indexes of $P_{*,*}$, as appearing
in Equation (\ref{eq:def-P-and-Q}), are switched with respect to
the usual Definition of Cartan-Eilenberg projective resolution (see
\cite[Definition 5.7.1]{Weibel94} for reference). This unusual choice
is made so that later results (see Equation (\ref{eq:long-exact-sequence}))
arise more naturally.
\end{rem}

For every $n\ge0$, the total cocomplex $\Tot\left(Q\right)^{*}$
of $Q^{*,*}$ satisfies
\begin{align}
\Tot\left(Q\right)^{n} & :=Q^{n-1,1}\oplus Q^{n,0}, & d^{n} & :\underset{\;\;\;\;\left(x,y\right)\;\;\;\;\,\longrightarrow\,\left(d_{h}^{n-1,1}\left(x\right)+d_{v}^{n,0}\left(y\right),d_{h}^{n,0}\left(y\right)\right)}{\Tot\left(Q\right)^{n}\:\to\:\:\;\;\;\;\;\;\;\;\Tot\left(Q\right)^{n+1}\:\;\;\;\;\;\;\;\;}.\label{eq:def-tot}
\end{align}
It is well known (see \cite[Theorem I.2.15]{McCleary85}) that, taking
vertical and horizontal filtrations of $Q^{*,*}$, one can obtain
two spectral sequences converging to $\Tot\left(Q\right)^{*}$. Since
$P_{*,*}$ is a Cartan-Eilenberg projective resolution, then these
spectral sequences can be simplified (see \cite[Paragraph 5.7.9]{Weibel94})
in order to obtain the following.
\begin{prop}
\label{prop:spectral-sequences}There exist two converging spectral
sequences $\lui{I}{E_{2}^{s,t}}\Rightarrow H^{s+t}\left(\Tot\left(Q\right)^{*}\right)$
and $\lui{II}{E_{2}^{s,t}}\Rightarrow H^{s+t}\left(\Tot\left(Q\right)^{*}\right)$
where, for every $s,t\ge0$ we define
\begin{align*}
\lui{I}{E_{2}^{s,t}} & :=H^{t}\left(\Ext_{\OFC}^{s}\left(CX_{*},M\right)\right), & \lui{II}{E_{2}^{s,t}} & :=\begin{cases}
\Ext_{\OFC}^{s}\left(\coker\left(f\right),M\right) & \text{if }t=0\\
\Ext_{\OFC}^{s}\left(\ker\left(f\right),M\right) & \text{if }t=1\\
0 & \text{else}
\end{cases}.
\end{align*}
Here $\Ext_{\OFC}^{s}\left(-,M\right)$ is the $s^{\text{th}}$ right
derived functor of $\Nat_{\OFC}\left(-,M\right)$ and $\Ext_{\OFC}^{s}\left(CX_{*},M\right)$
is the chain cocomplex obtained applying $\Ext_{\OFC}^{s}\left(-,M\right)$
to $CX_{*}$.
\end{prop}

\begin{proof}
Since $P_{*,*}$ is a Cartan-Eilenberg projective resolution of $CX_{*}$,
then it is an injective resolution in the opposite category. The two
spectral sequences in the statement are now exactly the two spectral
sequences appearing in \cite[Paragraph 5.7.9]{Weibel94} with $\boldsymbol{A}$
the cochain complex dual to $CX_{*}$ and $F=\Nat_{\OFC}\left(-,M\right)$.
\end{proof}
\begin{rem}
\label{rem:comparison-ext-groups}Let $F:\left(\R\Mod\right)^{\OFC^{\text{op}}}\to\R\OFC\Mod$
be the well known equivalence of categories already mentioned in Remark
\ref{rem:equivalence-of-index-cat-RC-mod}. There exist natural isomorphisms
$\Ext_{\OFC}^{s}\left(A,B\right)\cong\Ext_{\R\OFC}^{s}\left(F\left(A\right),F\left(B\right)\right)$
where $\Ext_{\R\OFC}^{s}$ denotes the $s^{\text{th}}$ $\Ext$ group.
This provides a more usual notion to interpret the results of Proposition
\ref{prop:spectral-sequences}.
\end{rem}

Lemmas \ref{lem:definition-f-ker-f} and \ref{lem:coker-f} allow
us to reinterpret the terms $\lui{II}{E_{2}^{s,t}}$ in terms of the
constant functor. The following allows us to further simplify these
terms.
\begin{lem}
\label{lem:relation-ext-lim}Let $P\in\mathcal{C}$, and let $\H$
be any fusions subsystem of $\F$ over $P$. The following holds for
every $n\ge0$
\[
\Ext_{\OFC}^{n}\left(\underline{\R}_{\OFC[\mathcal{H}]}^{\OFC},M\right)\cong\limn_{\OFC[\mathcal{H}]}\left(M\downarrow_{\OFC[\mathcal{H}]}^{\OFC}\right).
\]
\end{lem}

\begin{proof}
We know from \cite{LearyStancu07}, that there exist discrete groups
$H\ge P$ and $G'\ge S$ such that $\FAB{P}{H}=\H$ and $\FSG[G']=\F$.
Moreover, $P$ (resp. $S$) is the unique up to $H$-conjugacy (resp.
$G'$-conjugacy), maximal $p$-subgroup of $H$ (resp. $G'$) We write
$P\in\Syl_{p}\left(H\right)$ (resp. $S\in\Syl_{p}\left(G'\right)$)
to refer to this property.

Use the natural inclusion of $P$ into $S$ to define the amalgam
$G:=G'*_{P}H$. We know from \cite[Thereom 3.1]{ClellandParker10}
(see also \cite[Theorem 1]{Robinson07}) that $S\in\Syl_{p}\left(G\right)$
and that $\FSG=\left\langle \F,\H\right\rangle _{S}=\F$. Moreover,
we can view $H$ as a subgroup of $G$. With notation as in Remark
\ref{rem:comparison-ext-groups}, we can use Shapiro's isomorphism
(see \cite[Proposition 4.8]{Yal22}) to conclude that
\[
\Ext_{\R\OFC}^{n}\left(F\left(\underline{\R}_{\OFC[\mathcal{H}]}^{\OFC}\right),F\left(M\right)\right)\cong\Ext_{\R\OFC}^{n}\left(F\left(\underline{\R}\downarrow_{\OFC[\mathcal{H}]}^{\OFC}\right),F\left(M\downarrow_{\OFC[\mathcal{H}]}^{\OFC}\right)\right).
\]
The result follows from the well known isomorphism 
\[
\Ext_{\OFC}^{n}\left(\underline{\R}\downarrow_{\OFC[\mathcal{H}]}^{\OFC},M\downarrow_{\OFC[\mathcal{H}]}^{\OFC}\right)\cong\limn_{\OFC[\mathcal{H}]}\left(M\downarrow_{\OFC[\mathcal{H}]}^{\OFC}\right).
\]
\end{proof}
For every $\mathcal{E}\in\Lambda\cup\left\{ \F\right\} $, define
$M^{\mathcal{E}}:=\lim_{\OFC[\mathcal{E}]}\left(M\downarrow_{\OFC[\mathcal{E}]}^{\OFC}\right)$.
Since $\mathcal{C}$ contains $S'$ and is closed under taking overgroups,
then $M^{\F_{\boldsymbol{e}}}$ and $M^{\F_{\boldsymbol{2}}}$ can
be identified with their projections on $M\left(S'\right)$. Similarly
$M^{\F_{\boldsymbol{1}}}$ and $M^{\F}$ can be identified with their
projection on $M\left(S\right)$. The universal maps $M^{\F_{\boldsymbol{2}}}\to M^{\F_{\boldsymbol{e}}}$
and $M^{\F}\to M^{\F_{\boldsymbol{1}}}$ correspond (under this identification)
to the respective natural inclusions. Moreover, the universal map
$M^{\F_{\boldsymbol{1}}}\to M^{\F_{\boldsymbol{e}}}$ corresponds
to the restriction of the map $M\left(\overline{\iota_{S'}^{S}}\right):M\left(S\right)\to M\left(S'\right)$.
Here $\overline{\iota_{S'}^{S}}$ denotes the equivalence class in
$\OFC$ of the natural inclusion $\iota_{S'}^{S}:S'\hookrightarrow S$.
We denote by $M^{\F_{\boldsymbol{1}}}\overline{\iota_{S'}^{S}}$ the
image of the universal map $M^{\F_{\boldsymbol{1}}}\to M^{\F_{\boldsymbol{e}}}$
and by $a\overline{\iota_{S'}^{S}}$ the image of any $a\in M^{\F_{\boldsymbol{1}}}$
via this map. 
\begin{lem}
\label{lem:redefining-lim1}Let $f^{*}:\underset{\OFC}{\Nat}\left(CX_{0},M\right)\to\underset{\OFC}{\Nat}\left(CX_{1},M\right)$
be the image of $f$ (see Lemma \ref{lem:definition-f-ker-f}) via
the contraviariant $\Nat_{\OFC}\left(-,M\right)$ functor. The following
isomorphisms hold
\begin{align*}
\ker\left(f^{*}\right) & \cong M^{\F}, & \coker\left(f^{*}\right) & \cong M^{\F_{\boldsymbol{e}}}/\left(M^{\F_{\boldsymbol{1}}}\overline{\iota_{S'}^{S}}+M^{\F_{\boldsymbol{2}}}\right).
\end{align*}
\end{lem}

\begin{proof}
Define $M^{\boldsymbol{1},\boldsymbol{2}}:=M^{\F_{\boldsymbol{1}}}\oplus M^{\F_{\boldsymbol{2}}}$.
From Lemma \ref{lem:relation-ext-lim} and additivity of the $\underset{\OFC}{\Nat}\left(-,M\right)$
functor, we know that there exist isomorphisms
\begin{alignat*}{4}
h:\underset{\left(\alpha,\beta\right)\vphantom{\left(\alpha\left(\left[\Id_{S}\right]_{\F_{\boldsymbol{1}}}\right),\beta\left(\left[\Id_{S'}\right]_{\F_{\boldsymbol{2}}}\right)\right)}}{\underset{\OFC}{\Nat}\left(CX_{0},M\right)}\;\;\underset{\longrightarrow\vphantom{\left(\alpha\left(\left[\Id_{S}\right]_{\F_{\boldsymbol{1}}}\right),\beta\left(\left[\Id_{S'}\right]_{\F_{\boldsymbol{2}}}\right)\right)}}{\bjarrow\vphantom{\underset{\OFC}{\Nat}}}\underset{\left(\alpha\left(\left[\Id_{S}\right]_{\F_{\boldsymbol{1}}}\right),\beta\left(\left[\Id_{S'}\right]_{\F_{\boldsymbol{2}}}\right)\right)}{M^{\boldsymbol{1},\boldsymbol{2}}\vphantom{\underset{\OFC}{\Nat}}} &  & \hspace{0.5cm} & \text{and} & \hspace{1cm} &  & l:\underset{\gamma\vphantom{\gamma\left(\left[\Id_{S'}\right]_{\F_{\boldsymbol{e}}}\right)}}{\underset{\OFC}{\Nat}\left(CX_{1},M\right)}\;\underset{\longrightarrow\vphantom{\gamma\left(\left[\Id_{S'}\right]_{\F_{\boldsymbol{e}}}\right)}}{\bjarrow\vphantom{\underset{\OFC}{\Nat}}}\underset{\gamma\left(\left[\Id_{S'}\right]_{\F_{\boldsymbol{e}}}\right)}{M^{\F_{\boldsymbol{e}}}\vphantom{\underset{\OFC}{\Nat}}}.
\end{alignat*}
Define the morphism $g:M^{\boldsymbol{1},\boldsymbol{2}}\to M^{\F_{\boldsymbol{e}}}$
sending $\left(a,b\right)$ to $-a\overline{\iota_{S'}^{S}}+b$. It
is evident that $f^{*}=l^{-1}gh$. Since $h$ and $l$ are isomorphisms
it follows that $\ker\left(f^{*}\right)\cong\ker\left(g\right)$ and
$\coker\left(f^{*}\right)\cong\coker\left(g\right)$. The second isomorphism
follows.

Since $\F=\left\langle \F_{\boldsymbol{1}},\F_{\boldsymbol{2}}\right\rangle _{S}$,
then a simple calculation leads to

\[
\ker\left(g\right)=\left\{ \left(a,a\overline{\iota_{S'}^{S}}\right)\in M^{\boldsymbol{1},\boldsymbol{2}}\mid a\overline{\iota_{S'}^{S}}\in M^{\F_{\boldsymbol{2}}}\right\} \cong\left\{ a\in M^{\F_{\boldsymbol{1}}}\mid a\overline{\iota_{S'}^{S}}\in M^{\F_{\boldsymbol{2}}}\right\} =M^{\F}.
\]

The first isomorphism follows.
\end{proof}
As an immediate corollary we obtain the following.
\begin{cor}
\label{cor:redef-Hn-total-complex}Assume that $\limn_{\OFC[\mathcal{H}]}\left(M\downarrow_{\OFC[\mathcal{H}]}^{\OFC}\right)=0$
for every $n\ge1$ and every $\mathcal{H}\in\Lambda$. Then
\[
H^{n}\left(\Tot\left(Q\right)^{*}\right)\cong\begin{cases}
\ker\left(f^{*}\right)\cong M^{\F} & \text{if }n=0\\
\coker\left(f^{*}\right)\cong M^{\F_{\boldsymbol{e}}}/\left(M^{\F_{\boldsymbol{1}}}\overline{\iota_{S'}^{S}}+M^{\F_{\boldsymbol{2}}}\right) & \text{if }n=1\\
0 & \text{else}
\end{cases}
\]
\end{cor}

\begin{proof}
From Lemma \ref{lem:relation-ext-lim}, we know that the assumption
in the statement is equivalent to saying that the spectral sequence
$\lui{I}{E_{2}^{s,t}}\Rightarrow H^{s+t}\left(\Tot\left(Q\right)^{*}\right)$
of Proposition \ref{prop:spectral-sequences} is sharp. The result
follows from Lemma \ref{lem:redefining-lim1}.
\end{proof}
\medskip{}

Let $d_{2}^{n,1}:\lui{II}{E_{2}^{n,1}}\to\lui{II}{E_{2}^{n+2,0}}$
denote the non zero differentials of the spectral sequence $\lui{II}{E_{2}^{s,t}}\Rightarrow H^{s+t}\left(\Tot\left(Q\right)^{*}\right)$
of Proposition \ref{prop:spectral-sequences}. Since $\lui{II}{E_{2}^{s,t}}=0$
for every $t\not=0,1$, then (see \cite[Example I.1.D]{McCleary85})
there exists a well known long exact sequence of the form
\begin{equation}
\cdots\rightarrow H^{n+1}\left(\Tot\left(Q\right)^{*}\right)\to\Ext_{\OFC}^{n}\left(\ker\left(f\right),M\right)\stackrel{d_{2}^{n,1}}{\to}\Ext_{\OFC}^{n+2}\left(\coker\left(f\right),M\right)\to H^{n+2}\left(\Tot\left(Q\right)^{*}\right)\rightarrow\cdots.\label{eq:long-exact-sequence}
\end{equation}
Lemmas \ref{lem:definition-f-ker-f} and \ref{lem:coker-f} allow
us to rewrite the above exact sequence in terms of $\Nat_{\OFC}\left(\CFR,M\right)$
and $\limn[2]_{\OFC}\left(M\right)$. If $\limn_{\OFC[\mathcal{H}]}\left(M\downarrow_{\OFC[\mathcal{H}]}^{\OFC}\right)=0$
for every $n\ge1$ and every $\mathcal{H}\in\Lambda$, then we can
also apply Corollary \ref{cor:redef-Hn-total-complex} in order to
obtain the short exact sequence
\begin{equation}
0\to\limn[1]_{\OFC}\left(M\right)\to\coker\left(f^{*}\right)\stackrel{\varUpsilon}{\to}\Nat_{\OFC}\left(\CFR,M\right)\stackrel{d_{2}^{0,1}}{\to}\limn[2]_{\OFC}\left(M\right)\to0,\label{eq:short-exact-sequence}
\end{equation}
and, for every $n\ge1$, an isomorphisms of the form
\begin{equation}
\Ext_{\OFC}^{n}\left(\CFR,M\right)\cong\limn[n+2]_{\OFC}\left(M\right).\label{eq:isomorphims}
\end{equation}
The reminder of this section is dedicated to computing the morphism
$\varUpsilon$ of Equation (\ref{eq:short-exact-sequence}). More
precisely we prove that it can be taken as in Theorem \hyperref[thm:A]{A}. 

With notation as in Equation (\ref{eq:def-P-and-Q}), we define the
groups
\begin{align}
K:=K^{0,1} & :=\left\{ \left(\alpha,\beta\right)\in Q^{0,1}\oplus Q^{1,0}\mid d_{h}^{0,1}\left(\alpha\right)+d_{v}^{1,0}\left(\beta\right)=0\right\} ,\label{eq:def-K}\\
K^{2,0} & :=\left(\ker\left(d_{h}^{2,0}\right)\cap\ker\left(d_{v}^{2,0}\right)\right),\nonumber 
\end{align}
and
\begin{align}
I:=I^{0,1} & :=\left\{ \left(d_{v}^{0,0}\left(\gamma\right),d_{h}^{0,0}\left(\gamma\right)+\delta\right)\in Q^{0,1}\oplus Q^{1,0}\mid\left(\gamma,\delta\right)\in Q^{0,0}\oplus\ker\left(d_{v}^{1,0}\right)\right\} ,\label{eq:def-I}\\
I^{2,0} & :=d_{h}^{1,0}\left(\ker\left(d_{v}^{1,0}\right)\right)\nonumber 
\end{align}
The following holds.
\begin{lem}
\label{lem:definition-d2}For every $\left(\alpha,\beta\right)\in K$
let $\overline{\overline{\left(\alpha,\beta\right)}}\in K^{0,1}/I^{0,1}$
and $\overline{d_{h}^{1,0}\left(\beta\right)}\in K^{2,0}/I^{2,0}$
be the corresponding projections. There exist isomorphisms 
\begin{align*}
\lui{II}{E_{2}^{0,1}} & \cong K^{0,1}/I^{0,1} & \text{and} &  & \lui{II}{E_{2}^{2,0}} & \cong K^{2,0}/I^{2,0}.
\end{align*}
Moreover, identifying $\lui{II}{E_{2}^{0,1}}$ and $\lui{II}{E_{2}^{2,0}}$
with their images under these isomorphims, then
\[
d_{2}^{0,1}\left(\overline{\overline{\left(\alpha,\beta\right)}}\right):=\overline{d_{h}^{1,0}\left(\beta\right)}.
\]
\end{lem}

\begin{proof}
This is just a particular case of the well known description of spectral
sequences arising from double cocomplexes. See either \cite[Exercise 5.1.2]{Weibel94}
or \cite[Proposition 3.A.3]{PraderioPhD24} for the more general case
and a detailed proof.
\end{proof}
For every $\left(\alpha,\beta\right)\in\ker\left(d^{1}\right)$ let
$\overline{\left(\alpha,\beta\right)}\in H^{1}\left(\Tot\left(Q\right)^{*}\right)$
be its projection. With notation as in Equation (\ref{eq:def-tot}),
we have that $\ker\left(d^{1}\right)\subseteq K$ and that $\img\left(d^{0}\right)\subseteq I$.
This allows us to define the morphism
\begin{equation}
\Theta:\underset{\overline{\left(\alpha,\beta\right)}\vphantom{\overline{\overline{\left(\alpha,\beta\right)}}}}{H^{1}\left(\Tot\left(Q\right)^{*}\right)}\underset{\longrightarrow\vphantom{\overline{\overline{\left(\alpha,\beta\right)}}}}{\to\vphantom{H^{1}\left(\Tot\left(Q\right)^{*}\right)}}\underset{\overline{\overline{\left(\alpha,\beta\right)}}}{K/I\vphantom{H^{1}\left(\Tot\left(Q\right)^{*}\right)}}.\label{eq:def-Theta}
\end{equation}

\begin{lem}
\label{lem:cokernel-Theta}Assume that $\limn_{\OFC[\mathcal{H}]}\left(M\downarrow_{\OFC[\mathcal{H}]}^{\OFC}\right)=0$
for every $n\ge1$ and every $\mathcal{H}\in\Lambda$. Then $\coker\left(\Theta\right)\cong\limn[2]_{\OFC}\left(M\right)$.
\end{lem}

\begin{proof}
Identify $\lui{II}{E_{2}^{0,1}}$ with $K/I$ as in Lemma \ref{lem:definition-d2}.
From the exact sequence in Equation (\ref{eq:short-exact-sequence}),
it suffices to prove that $\img\left(\Theta\right)=\ker\left(d_{2}^{0,1}\right)$.
It is immediate from Equation (\ref{eq:def-Theta}), that $\img\left(\Theta\right)$
is equal to the projection $\overline{\overline{\ker\left(d^{1}\right)+I}}$
of $\ker\left(d^{1}\right)+I$ in $K/I$. On the other hand, we can
rewrite $\ker\left(d^{1}\right)+I$ as
\begin{align*}
\ker\left(d^{1}\right)+I & =\left\{ \left(\alpha,\beta+\delta\right)\in Q^{0,1}\oplus Q^{1,0}\mid d_{h}^{0,1}\left(\alpha\right)+d_{v}^{1,0}\left(\beta\right)=d_{v}^{1,0}\left(\delta\right)=d_{h}^{1,0}\left(\beta\right)=0\right\} ,\\
 & =\left\{ \left(\alpha,\beta\right)\in Q^{0,1}\oplus\left(\ker\left(d_{v}^{1,0}\right)+\ker\left(d_{h}^{1,0}\right)\right)\mid d_{h}^{0,1}\left(\alpha\right)+d_{v}^{1,0}\left(\beta\right)=0\right\} .
\end{align*}
It follows from Lemma \ref{lem:definition-d2}, that $\overline{\overline{\ker\left(d^{1}\right)+I}}=\ker\left(d_{2}^{0,1}\right)$.
This concludes the proof.
\end{proof}
\medskip{}

Before proceeding we need to make a brief parenthesis concerning the
Cartan-Eilenberg projective resolution $P_{*,*}$ of Equation (\ref{eq:def-P-and-Q}).

Let 
\begin{align*}
\CFR & \stackrel{\varepsilon^{k}}{\twoheadleftarrow}\ker\left(d_{0,1}^{v}\right)\stackrel{d_{1}^{k}}{\leftarrow}\cdots & , &  & \img\left(f\right) & \stackrel{\varepsilon^{i}}{\twoheadleftarrow}\img\left(d_{0,1}^{v}\right)\stackrel{d_{1}^{i}}{\leftarrow}\cdots & \text{and} &  & \coker\left(f\right) & \stackrel{\varepsilon^{c}}{\twoheadleftarrow}\coker\left(d_{0,1}^{v}\right)\stackrel{d_{1}^{c}}{\leftarrow}\cdots,
\end{align*}
be the chain complexes arising from $CX_{*}\stackrel{\varepsilon_{*}}{\twoheadleftarrow}P_{*,*}$
by restricting and taking quotients as appropriate. From the properties
of the Cartan-Eilenberg projective resolution (see \cite[Definition 5.7.1]{Weibel94}),
we know that these chain complexes are projective resolutions. Following
the usual arguments (see \cite[Proposition 6.24 and Theorem 10.45]{Rot79}),
we can use the above projective resolutions in combination with the
Horseshoe Lemma in order to build a Cartan-Eilenberg projective resolutions
$CX_{*}\stackrel{\eta_{*}}{\twoheadleftarrow}R_{*,*}$. More precisely
we can proceed as follows. For every $n\ge0$ define
\begin{align*}
R_{n,1} & :=\ker\left(d_{n,1}^{v}\right)\oplus\img\left(d_{n,1}^{v}\right), & R_{n,0} & :=\img\left(d_{n,1}^{v}\right)\oplus\coker\left(d_{n,1}^{v}\right),
\end{align*}
as well as the morphisms $d'{}_{n,1}^{v}:R_{n,1}\to R_{n,0}$ which
send every $\left(a,b\right)\in R_{n,1}$ to $\left(\left(-1\right)^{n}b,0\right)$.
The following hold trivially
\begin{align*}
\ker\left(d'{}_{n,1}^{v}\right) & =\ker\left(d_{n,1}^{v}\right), & \img\left(d'{}_{n,1}^{v}\right) & =\img\left(d_{n,1}^{v}\right), & \coker\left(d'{}_{n,1}^{v}\right) & =\coker\left(d_{n,1}^{v}\right).
\end{align*}
Let $\pi:CX_{0}\twoheadrightarrow\coker\left(f\right)$ denote the
natural projection. Since both $\img\left(d_{0,1}^{v}\right)$ and
$\coker\left(d_{0,1}^{v}\right)$ are projective, then we can take
morphisms $\sigma_{0}:\img\left(d_{0,1}^{v}\right)\to CX_{1}$ and
$\rho_{0}:\coker\left(d_{0,1}^{v}\right)\to CX_{0}$ such that $f\sigma_{0}=\varepsilon^{i}$
and that $\pi\rho_{0}=\varepsilon^{c}$. Define $\eta_{1}:R_{0,1}\to CX_{1}$
and $\eta_{0}:R_{0,0}\to CX_{0}$ by setting
\begin{align*}
\eta_{1}\left(\left(a,b\right)\right) & :=\varepsilon^{k}\left(a\right)+\sigma_{0}\left(b\right), & \eta_{0}\left(\left(b,\overline{c}\right)\right) & :=\varepsilon^{i}\left(b\right)+\rho_{0}\left(\overline{c}\right).
\end{align*}
Fix representatives $0\in\left[CX_{1}/\ker\left(f\right)\right]$
of the cosets $CX_{1}/\ker\left(f\right)$. For every $b\in\img\left(d_{0,1}^{v}\right)$,
there exist unique $x_{b}\in\ker\left(f\right)$ and $y_{b}\in\left[CX_{1}/\ker\left(f\right)\right]$
such that $\sigma_{0}\left(b\right)=x_{b}+y_{b}$. We conclude that
$\left(a,b\right)\in\ker\left(\eta_{1}\right)$ if and only $y_{b}=0$
and $\varepsilon^{k}\left(a\right)=-x_{b}$. Since $\img\left(\varepsilon^{k}\right)=\ker\left(f\right)$,
then we also deduce that, whenever $y_{b}=0$, there exists $a\in\ker\left(d_{0,1}^{v}\right)$
such that $\varepsilon^{k}\left(a\right)=-x_{b}$. In other words
$\left(a,b\right)\in\ker\left(\eta_{1}\right)$. We conclude that
the natural projection $R_{0,1}\twoheadrightarrow\img\left(d_{0,1}^{v}\right)$
restricts to a projection $\pi_{0}^{i}:\ker\left(\eta_{1}\right)\twoheadrightarrow\ker\left(\varepsilon^{i}\right)$.
Similarly we obtain that the natural projection $R_{0,0}\twoheadrightarrow\coker\left(d_{0,1}^{v}\right)$
restricts to a projection $\pi_{0}^{c}:\ker\left(\eta_{0}\right)\twoheadrightarrow\ker\left(\varepsilon^{c}\right)$.

Define
\begin{align*}
d'{}_{0,1}^{h} & :=\eta_{1}, & d'{}_{0,0}^{h} & :=\eta_{0}, & d_{0}^{i} & :=\varepsilon^{i}, & d_{0}^{c} & :=\varepsilon^{c}.
\end{align*}
For $n\ge1$ assume that the natural projections $R_{n-1,1}\twoheadrightarrow\img\left(d_{n-1,1}^{v}\right)$
and $R_{n-1,0}\twoheadrightarrow\coker\left(d_{n-1,1}^{v}\right)$
restrict, respectively, to epimorphisms $\pi_{n-1}^{i}:\ker\left(d'{}_{n-1,1}^{h}\right)\twoheadrightarrow\ker\left(d_{n-1}^{i}\right)$
and $\pi_{n-1}^{c}:\ker\left(d'{}_{n-1,0}^{h}\right)\twoheadrightarrow\ker\left(d_{n-1}^{c}\right)$.
Take morphisms $\sigma_{n}:\img\left(d_{n,1}^{v}\right)\to\ker\left(d'{}_{n,1}^{h}\right)$
and $\rho_{n}:\coker\left(d_{n,1}^{v}\right)\to\ker\left(d'{}_{n,1}^{h}\right)$
such that $\pi_{n-1}^{i}\sigma_{n}=d_{n}^{i}$ and that $\pi_{n-1}^{c}\rho_{n}=d_{n}^{c}$.
Define 
\begin{align*}
d'{}_{n,1}^{h}\left(\left(a,b\right)\right) & :=d_{n}^{k}\left(a\right)+\sigma_{n}\left(b\right), & d'{}_{n,0}^{h}\left(\left(b,\overline{c}\right)\right) & :=d_{n}^{i}\left(b\right)+\rho_{n}\left(\overline{c}\right).
\end{align*}
The same arguments as before show that the natural projections $R_{n,1}\twoheadrightarrow\img\left(d_{n,1}^{v}\right)$
and $R_{n,0}\twoheadrightarrow\coker\left(d_{n,1}^{v}\right)$ lead
to projections $\pi_{n}^{i}:\ker\left(d'{}_{n,1}^{h}\right)\twoheadrightarrow\ker\left(d_{n}^{i}\right)$
and $\pi_{n}^{c}:\ker\left(d'{}_{n,0}^{h}\right)\twoheadrightarrow\ker\left(d_{n}^{c}\right)$.
This allows us to define the differentials $d'{}_{n,1}^{h}$ and $d'{}_{n,0}^{h}$
inductively. It is now easy to see that $\eta_{1}$ and $\eta_{0}$
are surjective and that $\ker\left(d'{}_{n-,i}^{h}\right)=\img\left(d'{}_{n,i}^{h}\right)$
for every $n\ge1$ and $i=0,1$. We conclude that $CX_{*}\stackrel{\eta_{*}}{\twoheadleftarrow}R_{*,*}$
is a Cartan-Eilenberg projective resolution. 
\begin{notation}
\label{nota:redefinition-P}For the reminder of this section we take
the Cartan-Eilenberg resolution $CX_{*}\stackrel{\varepsilon_{*}}{\twoheadleftarrow}P_{*,*}$
of Equation (\ref{eq:def-P-and-Q}) to be equal to the Cartan-Eilenberg
resolution $CX_{*}\stackrel{\eta_{*}}{\twoheadleftarrow}R_{*,*}$
constructed above.
\end{notation}

We also introduce the following.
\begin{defn}
For every $n\ge0$, we define
\begin{align*}
Q_{k}^{n} & :=\Nat_{\OFC}\left(\ker\left(d_{n,1}^{v}\right),M\right), & Q_{i}^{n} & :=\Nat_{\OFC}\left(\img\left(d_{n,1}^{v}\right),M\right), & Q_{c}^{n} & :=\Nat_{\OFC}\left(\coker\left(d_{n,1}^{v}\right),M\right),
\end{align*}
and we identify them as subgroups of $Q^{n,1}$ and $Q^{n,0}$ as
appropriate.

We also denote by $d_{k}^{n}:Q_{k}^{n}\to Q_{k}^{n+1}$ the image
of $d_{n+1}^{k}:\ker\left(d_{n+1,1}^{v}\right)\to\ker\left(d_{n,1}^{v}\right)$
under the $\Nat_{\OFC}\left(-,M\right)$ functor. We define $d_{i}^{n}:Q_{i}^{n}\to Q_{i}^{n+1}$
and $d_{c}^{n}:Q_{c}^{n}\to Q_{c}^{n+1}$ analogously.

Finally, for every $\alpha\in Q^{n,1}$ and $\beta\in Q^{n,0}$ we
denote by $\alpha_{k}\in Q_{k}^{n}$, $\alpha_{i}\in Q_{i}^{n}\le Q^{n,1}$,
$\beta_{i}\in Q_{i}^{n}\le Q^{n,1}$ and $\beta_{c}\in Q_{c}^{n}$
the unique elements such that $\alpha=\alpha_{k}+\alpha_{i}$ and
$\beta=\beta_{i}+\beta_{c}$.
\end{defn}

With the given notation the following result is immediate.
\begin{lem}
With notation as in Equation (\ref{eq:def-P-and-Q}), the following
isomorphisms hold for every $n\ge0$
\begin{align*}
\coker\left(d_{v}^{n,0}\right) & \cong Q_{k}^{n}, & \ker\left(d_{v}^{n,0}\right) & \cong Q_{c}^{n}.
\end{align*}
\end{lem}

\medskip{}

Using Notation \ref{nota:redefinition-P} we can now prove the following
\begin{lem}
\label{lem:kernel-theta}There is an isomorphism $\ker\left(\Theta\right)\cong\limn[1]_{\OFC}\left(M\right)$.
\end{lem}

\begin{proof}
A simple calculation shows that
\[
I\cap\ker\left(d^{1}\right)=\left\{ \left(d_{v}^{0,0}\left(\gamma\right),d_{h}^{0,0}\left(\gamma\right)+\delta\right)\in Q^{0,1}\oplus Q^{1,0}\mid\left(\gamma,\delta\right)\in Q^{0,0}\oplus\left(\ker\left(d_{v}^{1,0}\right)\cap\ker\left(d_{h}^{1,0}\right)\right)\right\} .
\]
It follows that
\[
\ker\left(\Theta\right)\cong I\cap\ker\left(d^{1}\right)/\img\left(d^{1}\right)\cong\ker\left(d_{v}^{1,0}\right)\cap\ker\left(d_{h}^{1,0}\right).
\]
From the construction of $P_{*,*}$ given in Notation \ref{nota:redefinition-P},
we deduce that
\[
\ker\left(d_{v}^{1,0}\right)\cap\ker\left(d_{h}^{1,0}\right)=\Nat_{\OFC}\left(\coker\left(d_{0,1}^{v}\right),M\right)\cap\ker\left(d_{h}^{1,0}\right)=\ker\left(\left(d_{1}^{c}\right)^{*}\right),
\]
where $\left(d_{1}^{c}\right)^{*}$ denotes the image of $d_{1}^{c}:\coker\left(d_{1,1}^{v}\right)\to\coker\left(d_{0,1}^{v}\right)$
via the $\Nat_{\OFC}\left(-,M\right)$ functor. The result follows.
\end{proof}
Using the construction of Notation \ref{nota:redefinition-P} we also
obtain the following.
\begin{lem}
\label{lem:def-iso-Phi}Adopt the notation of Equations (\ref{eq:def-K})
and (\ref{eq:def-I}) and Lemma \ref{lem:definition-d2}, and let
$\phi:\ker\left(d_{k}^{1}\right)\bjarrow\Nat_{\OFC}\left(\CFR,M\right)$
be an isomorphism deriving from left exactness of the $\Nat_{\OFC}\left(-,M\right)$
functor. Then $\alpha_{k}\in\ker\left(d_{k}^{1}\right)$ for every
$\left(\alpha,\beta\right)\in K$ and the following is an isomorphism
\[
\Phi:\underset{\overline{\overline{\left(\alpha,\beta\right)}}}{K/I\vphantom{\Nat_{\OFC}\left(\CFR,M\right)}}\;\underset{\vphantom{\overline{\overline{\left(\alpha,\beta\right)}}}\longrightarrow}{\bjarrow\vphantom{\Nat_{\OFC}\left(\CFR,M\right)}}\underset{\phi\left(\alpha_{k}\right)\vphantom{\overline{\overline{\left(\alpha,\beta\right)}}}}{\Nat_{\OFC}\left(\CFR,M\right)}.
\]
\end{lem}

\begin{proof}
The vertical differentials $d_{v}^{n,0}$ send $Q_{i}^{n}\le Q^{n,0}$
isomorphically to $Q_{i}^{n}\le Q^{n,1}$ and satisfy $\ker\left(d_{v}^{n,0}\right)=Q_{c}^{n}$.
Define now
\begin{align*}
A & :=\left\{ \left(\alpha_{k},-\alpha_{k}\sigma_{1}\right)\in Q_{k}^{0}\oplus Q_{i}^{1}\le Q^{0,1}\oplus Q^{1,0}\mid\alpha_{k}d_{1}^{k}=0\right\} .
\end{align*}
It follows from the construction of $P_{*,*}$ in Notation \ref{nota:redefinition-P},
that the morphism $K/I\to A$ sending $\overline{\overline{\left(\alpha,\beta\right)}}$
to $\left(\alpha_{k},-\alpha_{k}\sigma_{1}\right)$ is an isomorphism.
In particular $\alpha_{k}\in\ker\left(d_{k}^{1}\right)$. Since the
projection of the elements of $A$ onto their first component leads
to an isomorphism $A\to\ker\left(d_{k}^{1}\right)$, then the result
follows. 
\end{proof}
Under some stronger assumptions, we also obtain the following isomorphism.
\begin{lem}
\label{lem:def-iso-Psi}Assume that $\limn[1]_{\OFC[\H]}\left(M\downarrow_{\OFC[\H]}^{\OFC}\right)=0$
for every $\H\in\Lambda$. The following is an isomorphism. 
\begin{align*}
\Psi: & \underset{\overline{\alpha}\vphantom{\overline{\left(\alpha\varepsilon_{1},0\right)}}}{\coker\left(f^{*}\right)\vphantom{H^{1}\left(\Tot\left(Q\right)^{*}\right)}}\underset{\vphantom{\overline{\left(\alpha\varepsilon_{1},0\right)}}\longrightarrow}{\vphantom{H^{1}\left(\Tot\left(Q\right)^{*}\right)}\bjarrow}\;\underset{\overline{\left(\alpha\varepsilon_{1},0\right)}}{H^{1}\left(\Tot\left(Q\right)^{*}\right)}.
\end{align*}
\end{lem}

\begin{proof}
The assumption immediately implies that $\ker\left(d_{h}^{1,0}\right)=\img\left(d_{h}^{0,0}\right)$.
It follows that, for every $\overline{\left(\alpha,\beta\right)}\in H^{1}\left(\Tot\left(Q\right)^{*}\right)$,
there exists a unique $\overline{\gamma}\in\ker\left(d_{h}^{0,1}\right)/\left(\ker\left(d_{h}^{0,1}\right)\cap\img\left(d_{v}^{0,0}\right)\right)$
such that $\overline{\left(\alpha,\beta\right)}=\overline{\left(\gamma,0\right)}$.
The result follows from left exactness of the $\Nat_{\OFC}\left(-,M\right)$
functor and the identities $f\varepsilon_{1}=\varepsilon_{0}d_{0,1}^{v}$.
\end{proof}
\begin{rem*}
The isomorphism of Lemma \ref{lem:def-iso-Psi} was already mentioned
(but not explicitly described) in Corollary \ref{cor:redef-Hn-total-complex}.
\end{rem*}
We can now prove the main result of this section.
\begin{prop}
\label{prop:theorem-0}Assume that $\limn_{\OFC[\mathcal{H}]}\left(M\downarrow_{\OFC[\mathcal{H}]}^{\OFC}\right)=0$
for every $n\ge1$ and every $\mathcal{H}\in\Lambda$. Then Equations
(\ref{eq:isomorphims}) and (\ref{eq:short-exact-sequence}) hold
with $\varUpsilon$ the morphism that, for every $\alpha\in\Nat_{\OFC}\left(CX_{1},M\right)$
sends the equivalence class $\overline{\alpha}\in\coker\left(f^{*}\right)$
to the restriction $\varUpsilon\left(\overline{\alpha}\right):=\alpha_{|\CFR}$
of $\alpha$ to $\CFR$.
\end{prop}

\begin{proof}
We only need to prove that Equation (\ref{eq:short-exact-sequence})
holds with $\varUpsilon$ as defined as the rest has already been
proven.

Let $\Theta$, $\Phi$ and $\Psi$ be as in Equation (\ref{eq:def-Theta})
and Lemmas \ref{lem:def-iso-Phi} and \ref{lem:def-iso-Psi} respectively.
The morphism $\varUpsilon$ of the statement can be alternatively
defined as the composition $\varUpsilon:=\Phi\Theta\Psi$. Since $\Psi$
and $\Phi$ are isomorphisms, then the result follows from Lemmas
\ref{lem:cokernel-Theta} and \ref{lem:kernel-theta}.
\end{proof}
As a corollary of Proposition \ref{prop:theorem-0} we can finally
prove Theorem \hyperref[thm:A]{A}.
\begin{proof}[Proof of Theorem A]
Let $\mathcal{E}\in\Lambda\cup\left\{ \F\right\} $. By assumption
$\mathcal{C}$ contains all $\mathcal{E}$-centric-radical subgroups
of $S$. It follows from \cite[Proposition 10.5]{Yal22} (see also
\cite[Corollary 3.6]{BLO03}) that, for every $n\ge1$
\[
\limn_{\OFC[\mathcal{E}]}\left(M^{*}\downarrow_{\OFC[\mathcal{E}]}^{\OF}\right)\cong\limn_{\OFc[\mathcal{E}]}\left(M^{*}\downarrow_{\OFc[\mathcal{E}]}^{\OF}\right).
\]
We conclude that the higher limits $\limn_{\OFC[\mathcal{H}]}\left(M^{*}\downarrow_{\OFC[\mathcal{H}]}^{\OF}\right)$
vanish for every $n\ge1$ and every $\mathcal{H}\in\Lambda$. The
result follows from Proposition \ref{prop:theorem-0} and the second
isomorphism of Lemma \ref{lem:redefining-lim1}.
\end{proof}

\section{Cohomological sharpness via the pruning method}\label{sec:sharpness_via_pruning}

Theorem \hyperref[thm:A]{A} allows us to study higher limits via
the use of the functor $\CFp$. Due to its relation with the amalgamation
method (see Proposition \ref{prop:amalgamation-method}), this functor
vanishes for a great number of known exotic fusion systems. The following
gives sufficient conditions for $\CFp$ to vanish.
\begin{lem}
\label{lem:vanishing-CFR}Adopt Notation \ref{nota:C-included-in-notation}
and let $P\in\mathcal{C}$. Assume that, for every $Q\cong_{\F}P$,
then 
\[
\Hom_{\F_{\boldsymbol{e}}}\left(Q,S'\right)=\Hom_{\F_{\boldsymbol{1}}}\left(Q,S'\right)\cap\Hom_{\F_{\boldsymbol{2}}}\left(Q,S'\right),
\]
and that either of the following holds:
\begin{enumerate}
\item \label{enu:hom-included}$\Hom_{\F}\left(P,S'\right)=\Hom_{\mathcal{\H}}\left(P,S'\right)$
for some $\H\in\left\{ \F_{\boldsymbol{1}},\F_{\boldsymbol{2}}\right\} $,
or
\item \label{enu:aut-included}$\Aut_{\F}\left(P\right)=\Aut_{\F_{\boldsymbol{2}}}\left(P\right)$
and, for every $P\not\cong_{\F_{\boldsymbol{e}}}Q\cong_{\F}P$, then
$\Hom_{\F_{\boldsymbol{e}}}\left(Q,S'\right)=\Hom_{\F_{\boldsymbol{2}}}\left(Q,S'\right)$
(equivalently $\Hom_{\F_{\boldsymbol{2}}}\left(Q,S'\right)\subseteq\Hom_{\F_{\boldsymbol{1}}}\left(Q,S'\right)$).
\end{enumerate}
Then $\textnormal{\ensuremath{\Rep_{\F}\left(P,\Lambda\right)}}$
is a tree. In particular, $\CFR\left(P\right)=0$ for any commutative
ring $\R$.
\end{lem}

\begin{proof}
Since $\F=\left\langle \F_{\boldsymbol{1}},\F_{\boldsymbol{2}}\right\rangle _{S}$
then $\Rep_{\F}\left(P,\Lambda\right)$ is connected. For every $\varphi\in\Hom_{\F}\left(P,S'\right)$,
there exists a unique edge (namely $\left[\varphi\right]_{\F_{\boldsymbol{e}}}$)
connecting the vertices $\left[\varphi\right]_{\F_{\boldsymbol{1}}}$
and $\left[\varphi\right]_{\F_{\boldsymbol{2}}}$ of $\textnormal{\textnormal{\ensuremath{\Rep_{\F}\left(P,\Lambda\right)}}}$.

Assume that Condition (\ref{enu:hom-included}) holds. Then $\textnormal{\ensuremath{\Rep_{\F}\left(P,\H\right)}}$
has a single element. Since $\textnormal{\textnormal{\ensuremath{\Rep_{\F}\left(P,\Lambda\right)}}}$
is a bipartite graph, one of the partitions has a single vertex and
there is at most one edge between any two vertices, then it is a tree. 

Assume that Condition (\ref{enu:aut-included}) holds. Let $\varphi\in\Hom_{\F}\left(P,S'\right)$
such that $\varphi\left(P\right)\not\cong_{\F_{\boldsymbol{e}}}Q$.
Then $\left[\varphi\right]_{\F_{\boldsymbol{e}}}=\left[\varphi\right]_{\F_{\boldsymbol{2}}}$.
In particular, the vertex $\text{\ensuremath{\left[\varphi\right]_{\F_{\boldsymbol{2}}}}}$
of $\textnormal{\textnormal{\ensuremath{\Rep_{\F}\left(P,\Lambda\right)}}}$
is connected to a single other vertex in $\textnormal{\ensuremath{\Rep_{\F}\left(P,\Lambda\right)}}$.
Let $X\subseteq\textnormal{\ensuremath{\Rep_{\F}\left(P,\boldsymbol{\Lambda}\right)}}$
be the full subgraph with vertexes $V\left(X\right)=\Rep_{\F}\left(P,\F_{\boldsymbol{1}}\right)\sqcup W$
where
\begin{align*}
W & :=\left\{ \left[\varphi\right]_{\F_{\boldsymbol{2}}}\in\Rep_{\F}\left(P,\F_{\boldsymbol{2}}\right)\mid\varphi\left(P\right)\cong_{\F_{\boldsymbol{e}}}P\right\} =\left\{ \left[\varphi\right]_{\F_{\boldsymbol{2}}}\in\Rep_{\F}\left(P,\F_{\boldsymbol{2}}\right)\mid\varphi\left(P\right)=P\right\} .
\end{align*}
From the above, then $\textnormal{\ensuremath{\Rep_{\F}\left(P,\Lambda\right)}}$
is a tree if and only if $X$ is also a tree. From the first half
of the assumption, we know that $W=\left\{ \left[\iota_{P}^{S'}\right]_{\F_{\boldsymbol{2}}}\right\} $.
The same arguments used in Part (\ref{enu:hom-included}) prove that
$X$ is a tree. The result follows.
\end{proof}
The following provides us with sufficient conditions for the base
assumption of Lemma \ref{lem:vanishing-CFR} to be satisfied.
\begin{lem}
\label{lem:intersection-of-fusion-systems.}Let $\H\subseteq\F$ be
a saturated fusion systems over a $p$-group $S$, let $P\le S$ be
fully $\F$-normalized and define $\mathcal{N}:=N_{\F}\left(P\right)\cap\H$.
Then $N_{\H}\left(P\right)\subseteq\mathcal{N}$ and the morphisms
of both fusion systems coincide in $\F$-centrics. In particular
\[
\OFC[\mathcal{N}]=\OFC[N_{\H}\left(P\right)],
\]
for every family $\mathcal{C}$ of $\F$-centric subgroups of $S$.
\end{lem}

\begin{proof}
Since $\H\subseteq\F$, then the inclusion $N_{\H}\left(P\right)\subseteq\mathcal{N}$
follows.

Let $S':=N_{S}\left(P\right)$, let $Q\le S'$ be $\F$-centric, let
$\varphi:Q\to S'$ be a morphism in $\mathcal{N}$ and let $\hat{\varphi}:QP\to S'$
be a morphism in $N_{\F}\left(P\right)$ lifting $\varphi$. Let $P\le R\le QP$
be a maximal subgroup such that the restriction $\psi:R\to S'$ of
$\hat{\varphi}$ to $R$ is a morphism in $\H$. Such a group exists
because $\varphi$ is a morphism in $\H$. Since $P$ is $\F$-centric
then $R$ is also $\F$-centric and, in particular, it is $\H$-centric.
It follows from \cite[Proposition 4.4]{Stancu03}, that $\psi$ can
be lifted to a morphism $\hat{\psi'}:N_{\psi}\to S$ in $\H$. Since
$\iota_{S'}^{S}\hat{\varphi}$ lifts $\psi$, then $N_{QP}\left(R\right)\le N_{\psi}$.
We conclude that $\psi$ can be lifted to a morphism $\hat{\psi}:N_{QP}\left(R\right)\to S$
in $\H$.

Denote by $\theta$ the restriction of $\hat{\varphi}$ to $N_{QP}\left(R\right)$.
Both $\hat{\psi}$ and $\iota_{S'}^{S}\theta$ lift $\iota_{S'}^{S}\psi$.
Since $R$ is $\F$-centric and $\F$ is saturated, then there exists
$x\in S$ such that $\theta\left(y\right)=\left(\hat{\psi}\left(y\right)\right)^{x}$
for every $y\in N_{QP}\left(R\right)$ (see \cite[Theorem 4.9]{Link07}
for detail). It follows that $\iota_{S'}^{S}\theta$ (and therefore
$\theta$) is a morphism in $\H$. From maximality of $R$ we conclude
that $R=N_{QP}\left(R\right)=QP$ and, therefore, $\theta=\hat{\varphi}$.
In particular $\hat{\varphi}$ is a morphism in $\H$ and, therefore,
in $N_{\H}\left(P\right)$. It follows that $\varphi$ is also a morphism
in $N_{\H}\left(P\right)$ thus concluding the proof.
\end{proof}
\begin{cor}
\label{cor:whole-vanishing-of-CFR}Let $\H\subseteq\F$ be saturated
fusion systems over a $p$-group $S$, let $\mathcal{C}$ be a family
of subgroups of $S$ closed under $\F$-conjugation and taking overgroups,
let $P\in\mathcal{C}$ be fully $\F$-normalized and define $\Lambda:=\left\{ \H,N_{\F}\left(P\right),N_{\H}\left(P\right)\right\} $.
Assume that $\F=\left\langle \H,\Aut_{\F}\left(P\right)\right\rangle _{S}$.
Then $\Rep_{\F}\left(Q,\Lambda\right)$ is a tree for every $Q\not<_{\F}P$.
In particular $\CFR\left(Q\right)=0$ for every such $Q$ and commutative
ring $\R$ and, if $P$ is minimal in $\mathcal{C}$ (under inclusion),
then $\CFR=0$.
\end{cor}

\begin{proof}
In order to simplify notation we define $\F_{\boldsymbol{1}}:=\H$,
$\F_{\boldsymbol{2}}:=N_{\F}\left(P\right)$, $\F_{\boldsymbol{e}}:=N_{\H}\left(P\right)$
and $S':=N_{S}\left(P\right)$. Notice that this notation is consistent
with Notation \ref{nota:C-included-in-notation}.

Let $Q$ be as in the statement. If $Q\not\cong_{\F}P$, then $\Hom_{\F}\left(Q,S\right)=\Hom_{\F_{\boldsymbol{1}}}\left(Q,S\right)$.
In this case the result follows from Lemmas \ref{lem:vanishing-CFR}(\ref{enu:hom-included})
and \ref{lem:intersection-of-fusion-systems.}. We can therefore assume
without loss of generality that $Q\cong_{\F}P$.

If $Q\not\le S'$ then $\Hom_{\F_{\boldsymbol{e}}}\left(Q,S'\right)=\Hom_{\F_{\boldsymbol{2}}}\left(Q,S'\right)=\emptyset$.
Assume that $P\not=Q\le S'$. Let $\varphi\in\Hom_{\F_{\boldsymbol{2}}}\left(Q,S'\right)$
and let $\hat{\varphi}:QP\to S'$ be a morphism in $\F_{\boldsymbol{2}}$
lifting $\varphi$. Since $P\lneq QP$, then it is not $\F$-conjugate
to $P$. It follows that $\hat{\varphi}$ (and therefore $\varphi$)
is a morphism in $\F_{\boldsymbol{1}}$. We conclude from Lemmas \ref{lem:vanishing-CFR}(\ref{enu:aut-included})
and \ref{lem:intersection-of-fusion-systems.}, that $\Rep_{\F}\left(P,\Lambda\right)$
is a tree. It follows from Lemma \ref{lem:functor-to-graph} that
$\Rep_{\F}\left(Q,\Lambda\right)$ is a tree. This concludes the proof.
\end{proof}
As a consequence of Corollary \ref{cor:whole-vanishing-of-CFR} we
obtain the following noteworthy result.
\begin{prop}
\label{prop:sharpness-2-essentials}Let $\F$ be a saturated fusion
system over a $p$-group $S$, let $M^{*}:\OF^{\operatorname{op}}\to\Fp\mod{}$
be the contravariant part of a Mackey functor over $\F$ with coefficients
in $\Fp$ and let $P,Q\in\Fcr$ be fully $\F$-normalized. Assume
that $P\trianglelefteq S$, that $\F=\left\langle N_{\F}\left(P\right),\Aut_{\F}\left(Q\right)\right\rangle _{S}$
and that $Q$ is minimal (under inclusion) in $\Fcr$. Then $\limn_{\OFc}\left(M^{*}\downarrow_{\OFc}^{\OF}\right)=0$
for every $n\ge2$. In particular, $\F$ satisfies cohomological sharpness.
\end{prop}

\begin{proof}
Cohomological sharpness follows from the first part of the statement
and \cite[Lemma 3.6]{Pra26}.

Define the fusion systems
\begin{align*}
\F_{\boldsymbol{1}} & :=N_{\F}\left(P\right), & \F_{\boldsymbol{2}} & :=N_{\F}\left(Q\right), & \F_{\boldsymbol{e}} & :=N_{\F_{\boldsymbol{1}}}\left(Q\right),
\end{align*}
the set $\Lambda=\left\{ \F_{\boldsymbol{1}},\F_{\boldsymbol{2}},\F_{\boldsymbol{e}}\right\} $
and let $\mathcal{C}$ be the smallest family of subgroups of $S$
containing $\Fcr$ and closed under $\F$-conjugates and taking overgroups.
This notation is consistent with Notation \ref{nota:C-included-in-notation}.

It follows from Lemma \ref{lem:centric-radical-contained}, that $\Fcr_{x}\subseteq\mathcal{C}$
for every $x\in\Ob\left(\T\right)$. From Lemma \ref{lem:sharpness-for-normalizers},
we know that $\F_{x}$ satisfies sharpness. From Corollary \ref{cor:whole-vanishing-of-CFR}
(applied with $P=Q$), we know that $\CFp=0$. The result follows
from applying Theorem \hyperref[thm:A]{A}.
\end{proof}
\begin{rem}
The above proposition can, in fact, be proven using \cite[Theorem B]{Pra26}
instead of Theorem \hyperref[thm:A]{A}. However, in doing so, the
groups $G_{x}$ realizing each $\F_{x}$ need to be the ones given
by \cite[Proposition 4.3]{BCGLO05}.
\end{rem}

We are now ready to prove Theorem \hyperref[thm:B]{B}.
\begin{proof}[\foreignlanguage{english}{Proof of Theorem B}]
Assume that $\left|S\right|\ge p^{4}$.

If $\left|\mathcal{A}\right|=0$ then $\F=\H$ and the result is obvious.
Fix $n\ge1$ and assume that $\F\not=\H$ and that the statement holds
for $\left|\mathcal{A}\right|=n-1$. Fix $P\in\mathcal{A}$ and define
\[
\F_{\boldsymbol{1}}:=\left\langle \H,\Aut_{\F}\left(Q\right)\mid Q\in\mathcal{A}\backslash P\right\rangle _{S}.
\]
From Proposition \ref{prop:pruning-method}, we know that $\F_{\boldsymbol{1}}$
is saturated. It follows from induction hypothesis, that $\F_{\boldsymbol{1}}$
satisfies cohomological sharpness whenever $\H$ does.

If $C_{S}\left(Q\right)\not\le Q$ for every $Q\lneq P$ then $\Fc_{\boldsymbol{1}}=\Fc$.
If $P\cong p_{+}^{1+2}$, then it follows from \cite[Lemma 5.20]{GraPar25},
that $\Fc_{\boldsymbol{1}}=\Fc$. In both cases $\Fcr_{\boldsymbol{1}}\subseteq\Fc$.

By taking an $\F_{\boldsymbol{1}}$-conjugate if necessary, we can
assume without loss of generality that $P$ is fully $\F$-normalized.
Define the fusion systems $\F_{\boldsymbol{2}}:=N_{\F}\left(P\right)$
and $\F_{\boldsymbol{e}}:=N_{\F_{\boldsymbol{1}}}\left(P\right)$.
Since $\F_{\boldsymbol{1}}$ and $\F$ are saturated, it follows from
Alperin's fusion theorem that $P\not\in\Fc$ implies $\F_{\boldsymbol{1}}=\F$.
We conclude that $P$ is $\F$-centric. From Lemma \ref{lem:sharpness-for-normalizers},
we know that $\F_{\boldsymbol{2}}$ and $\F_{\boldsymbol{e}}$ satisfy
cohomological sharpness. From Lemma \ref{lem:centric-radical-contained},
we know that $\Fcr_{\boldsymbol{2}},\Fcr_{\boldsymbol{e}}\subseteq\Fc$.

Define the set $\Lambda:=\left\{ \F_{\boldsymbol{1}},\F_{\boldsymbol{2}},\F_{\boldsymbol{e}}\right\} $
and the family of groups $\mathcal{C}=\Fc$. This notation is consistent
with Notation \ref{nota:C-included-in-notation}. From Corollary \ref{cor:whole-vanishing-of-CFR},
then $\CFp=0$. If $\H$ satisfies cohomological sharpness, then we
can apply Theorem \hyperref[thm:A]{A} to deduce that $\limi_{\OFc}\left(H^{j}\left(-;\Fp\right)\right)=0$
for every $i\ge2$ and every $j\ge0$. The result follows from \cite[Lemma 3.6]{Pra26}.

Assume that $\left|S\right|\le p^{3}$.

Redefine $\H:=N_{\F}\left(S\right)$ and $\mathcal{A}:=\Fc\backslash\left\{ S\right\} $.
From Alperin's fusion theorem, we know that $\F=\left\langle \H,\Aut_{\F}\left(Q\right)\mid Q\in\mathcal{A}\backslash P\right\rangle _{S}$.
From Lemma \ref{lem:sharpness-for-normalizers}, we know that $\H$
satisfies cohomological sharpness. By construction, every $P\in\mathcal{A}$
is abelian. In the arguments above we have used that $\left|S\right|\ge p^{4}$
only to apply \cite[Lemma 5.20]{GraPar25} and only when $P\cong p_{+}^{1+2}$.
Thus, after this redefinition, the same arguments as before apply.
It follows that $\F$ satisfies cohomological sharpness thus concluding
the proof. 
\end{proof}
During the above we have proven the following.
\begin{cor}
\label{cor:sharpness-groups-of-order-p3}Let $\F$ be a saturated
fusion system over any $p$-group $S$ of order $\left|S\right|\le p^{3}$.
Then $\F$ satisfies cohomological sharpness.
\end{cor}

\begin{rem}
Corollary \ref{cor:sharpness-groups-of-order-p3} is a special case
of \cite[Theorem C]{DiazPark15}. However the methods used to prove
it are drastically different.
\end{rem}

\subsection{Cohomological sharpness for polynomial fusion systems}\label{subsec:Polynomial-fs}

Let $p$ be a prime, let $q$ be a power of $p$ and let $1\le n\le p$.
We define $V_{n}\left(q\right)$ as the the group of homogenous polynomials
over $\Fp[q]$ on $2$ variables and degree $n$ and $\Lambda\left(q\right):=\Hom_{\Fp[q]}\left(V_{p}\left(q\right),\Fp[q]\right)$
as its dual. There exists a well known right action of $\GL_{2}\left(q\right)$
on $V_{n}\left(q\right)$ given by setting
\[
\left(f\cdot A\right)\left(x,y\right):=f\left(ax+by,cx+dy\right),
\]
for every $A:=\begin{pmatrix}a & b\\
c & d
\end{pmatrix}\in\GL_{2}\left(q\right)$, every $f\in V_{n}\left(q\right)$ and every $x,y\in\Fp[q]$. This
in turn leads to an action of $\GL_{2}\left(q\right)$ on $\Lambda\left(q\right)$
obtained by setting
\[
\left(\phi\cdot A\right)\left(f\right):=\phi\left(f\cdot A^{-1}\right),
\]
for every $A\in\GL_{2}\left(q\right)$, every $f\in V_{p}\left(q\right)$
and every $\phi\in\Lambda\left(q\right)$. This allows us to define
the following groups
\begin{align*}
P_{n}\left(q\right) & :=V_{n}\left(q\right)\rtimes\GL_{2}\left(q\right), & S_{n}\left(q\right) & \in\Syl_{p}\left(P_{n}\left(q\right)\right),\\
P_{\Lambda}\left(q\right) & :=\Lambda\left(q\right)\rtimes\GL_{2}\left(q\right), & S_{\Lambda}\left(q\right) & \in\Syl_{p}\left(P_{\Lambda}\left(q\right)\right).
\end{align*}

\begin{lem}
\label{lem:polynomial-characteristic}View $V:=V_{n}\left(q\right)$
and $V_{\Lambda}:=\Lambda\left(q\right)$ as subgroups of $S:=S_{n}\left(q\right)$
and $S_{\Lambda}:=S_{\Lambda}\left(q\right)$ respectively. The group
$V_{\Lambda}$ is characteristic in $S_{\Lambda}$ and the group $V$
is characterisitc in $S$ for every $1\le n\le p$.
\end{lem}

\begin{proof}
This is proven in \cite[Lemmas 2.8(4), 2.15(6) and 2.16(8)]{GPSB26}.
\end{proof}
Let $U\le\SL_{2}\left(q\right)$ be the subgroup of lower triangular
matrices. Viewing $U$ as a subgroup of both $S_{n}\left(q\right)$
and $S_{\Lambda}\left(q\right)$ we can further define the groups
\begin{align*}
R_{\Lambda} & :=UZ\left(S_{\Lambda}\left(q\right)\right) & R & :=UZ\left(S_{n}\left(q\right)\right), & Q & :=UZ_{2}\left(S_{n}\left(q\right)\right),
\end{align*}
where $Z_{2}\left(S\right)$ denotes the second center of a group
$S$ (see Notation \ref{nota:central-series}).

In \cite{GPSB26}, Grazian, Parker, Semeraro and van Beek (GPSB) study
the fusion systems $\F$ over $S_{n}\left(q\right)$ and $S_{\Lambda}\left(q\right)$
satisfying $O_{p}\left(\F\right)=\left\{ 1\right\} $. In particular,
they describe three fusion systems $\F^{*}\left(n,q,R\right)$, $\F^{*}\left(n,q,Q\right)$,
$\F^{*}\left(n,q,R\right)_{P}$ over $S_{n}\left(q\right)$ and a
fusion system $\F_{\Lambda}^{*}\left(q\right)$ over $S_{\Lambda}\left(q\right)$
and give the following definition.
\begin{defn}
\label{def:polynomial-fusion-system}A saturated fusion system $\F$
is a polynomial fusion system if there exists $1\le n\le p$ such
that $\F$ is subsystem of either $\F^{*}\left(n,q,R\right)$, $\F^{*}\left(n,q,Q\right)$,
$\F^{*}\left(n,q,R\right)_{P}$ or $\F_{\Lambda}^{*}\left(q\right)$,
has index prime to $p$ with respect to it and satisfies $O_{p}\left(\F\right)=\left\{ 1\right\} $.

Furthermore, GPSB provide the following description of the centric-radical
subgroups.
\end{defn}

\begin{lem}
\label{lem:polynomial-centric-radical}View $V:=V_{n}\left(q\right)$
and $V_{\Lambda}:=\Lambda\left(q\right)$ as subgroups of $S:=S_{n}\left(q\right)$
and $S_{\Lambda}:=S_{\Lambda}\left(q\right)$ respectively. Then 
\begin{align*}
\left(\F_{\Lambda}^{*}\left(q\right)\right)^{\operatorname{cr}} & =\left\{ S_{\Lambda},V_{\Lambda}\right\} \cup R_{\Lambda}^{S},\\
\left(\F^{*}\left(n,q,X\right)\right)^{\operatorname{cr}} & =\left\{ S,V\right\} \cup X^{S},\\
\left(\F^{*}\left(n,q,R\right)_{P}\right)^{\operatorname{cr}} & =\left\{ S\right\} \cup R^{S},
\end{align*}
where $X\in\left\{ Q,R\right\} $, and $Y^{S}$ denotes the set of
subgroups of $S$ that are $S$-conjugate to $Y$.
\end{lem}

\begin{proof}
The centric-radicals for $\F_{\Lambda}^{*}\left(q\right)$ and $\F^{*}\left(n,q,X\right)$
are given in \cite[Proposition 4.2]{GPSB26}. On the other hand, $\F^{*}\left(n,q,R\right)_{P}$
is obtained from $\F^{*}\left(n,q,R\right)$ by ``pruning'' the
automorphisms of $V$ as in Proposition \ref{prop:pruning-method}
(see \cite[Notation 4.4]{GPSB26}). From Alperins fusion theorem and
the description of $\F^{*}\left(n,q,R\right)$-centric-radical groups,
we deduce that all the $\F^{*}\left(n,q,R\right)_{P}$-automorphisms
of $V$ can be lifted to $\F^{*}\left(n,q,R\right)_{P}$-automorphisms
of $S$. On the other hand, the $\F^{*}\left(n,q,R\right)_{P}$-automorphism
of every other $\F^{*}\left(n,q,R\right)_{P}$-centric element coincide
with its $\F^{*}\left(n,q,R\right)$-automorphisms. The result follows.
\end{proof}
\begin{rem}
The last identity of Lemma \ref{lem:polynomial-centric-radical} is
also noted (with less explanation) in the first paragraph of the proof
of \cite[Proposition 4.5]{GPSB26}.
\end{rem}

The centric-radical groups for all the other polynomial fusion systems
are given by the following known result.
\begin{lem}[{\foreignlanguage{english}{\cite[Lemma 2.2(a)]{OliverRuiz21}}}]
\label{lem:index-p-prime-preserves-centric-radicals}Let $\F$ be
a saturated fusion system over a $p$-group $S$ and let $\H\subseteq\F$
be a fusion subsystem of index prime to $p$. Then $\Fcr=\Fcr[H]$.
\end{lem}

We can now prove the following
\begin{prop}
\label{prop:sharpness-polynomial}Let $\F$ be a polynomial fusion
system (in the sense of Definition \ref{def:polynomial-fusion-system})
and let $M^{*}:\OF^{\operatorname{op}}\to\Fp\mod{}$ be the contravariant
part of any Mackey functor over $\F$ with coefficients in $\Fp$.
Then $\limn_{\OFc}\left(M^{*}\downarrow_{\OFc}^{\OF}\right)=0$ for
every $n\ge2$. In particular, $\F$ satisfies cohomological sharpness
(see \cite[Lemma 3.6]{Pra26}).
\end{prop}

\begin{proof}
If $\F$ is a fusion system over $S_{n}\left(q\right)$, then let
$S:=S_{n}\left(q\right)$ and $V:=V_{n}\left(q\right)$ (seen as a
subgroup of $S$). Lemmas \ref{lem:polynomial-centric-radical} and
\ref{lem:index-p-prime-preserves-centric-radicals} ensure that, up
to $S$-conjugacy, there exists exactly one $\F$-centric-radical
subgroup different from both $S$ and $V$. Denote this subgroup by
$X$. If $\F$ is a fusion system over $S_{\Lambda}\left(q\right)$,
then let $S:=S_{\Lambda}\left(q\right)$ and $V:=\Lambda\left(q\right)$
(seen as a subgroup of $S$). Define $X$ as before. We know from
Lemmas \ref{lem:polynomial-centric-radical} and \ref{lem:index-p-prime-preserves-centric-radicals},
that $X$ is minimal in $\Fcr\subseteq\left\{ S,V\right\} \cup X^{S}$.
We know from Lemma \ref{lem:polynomial-characteristic}, that $V$
is characteristic in $S$. We deduce from Alperin's fusion theorem,
that $\F=\left\langle N_{\F}\left(V\right),\Aut_{\F}\left(X\right)\right\rangle _{S}$.
The result follows from applying Proposition $\ref{prop:sharpness-2-essentials}$
with $P=V$ and $Q=X$.
\end{proof}

\subsection{Cohomological sharpness for saturated fusion systems over $p$-groups
of maximal nihlpotency class}

Our goal during this section is that of proving the following.
\begin{prop}
\label{prop:sharpness-maximal-nihlpotency-class}Let $S$ be a $p$-group
$S$ of maximal nihlpotency class and let $\F$ be a saturated fusion
system over $S$. Then $\F$ satisfies cohomological sharpness.
\end{prop}

In order to prove Proposition \ref{prop:sharpness-maximal-nihlpotency-class}
we make use of the classification of all such fusion systems given
in \cite{GraPar25}. A key concept during this classification is that
of exceptional group and of Pearl.
\begin{notation}
\label{nota:central-series}Let $S$ be a group. We use the following
usual notation for the elements of the lower and upper central series:
\begin{itemize}
\item $\gamma_{2}\left(S\right):=\left[S,S\right]$ and for every $n\ge3$,
we define $\gamma_{n}\left(S\right)=\left[\gamma_{n-1}\left(S\right),S\right]$.
\item $Z_{1}\left(S\right):=Z\left(S\right)$ and, for every $n\ge2$, we
define $Z_{n}\left(S\right)$ as the pre-image in $S$ of $Z\left(S/Z_{n-1}\left(S\right)\right)$.
\end{itemize}
\end{notation}

\begin{defn}
Let $S$ be a $p$-group of maximal nihlpotency class and order at
least $p^{4}$. We define $\gamma_{1}:=C_{S}\left(\gamma_{2}\left(S\right)/\gamma_{4}\left(S\right)\right)$
and say that $S$ is exceptional if $\gamma_{1}\left(S\right)\not=C_{S}\left(Z_{2}\right)$.
\end{defn}

\begin{defn}
A Pearl of a saturated fusion system $\F$ is an $\F$-essential subgroup
$P\le S$ such that either $P\cong C_{p}\times C_{p}$ is elementary
abelian of rank $2$ or $P\cong p_{+}^{1+2}$ is extraspecial of order
$p^{3}$ and exponent $p$.
\end{defn}

With this notation we can state the following result of Grazian and
Parker.
\begin{prop}[{\foreignlanguage{english}{\cite[Theorem D]{GraPar25}}}]
\label{prop:classification-essentials-maximal-class}Let $\F$ be
a saturated fusion system over a $p$-group $S$ of maximal nihlpotency
class and let $E\in\mathcal{E}\left(\F\right)$ be a $\F$-essential
subgroup of $S$. Assume that $\left|S\right|\ge p^{4}$. Then either
$E\in\left\{ \gamma_{1}\left(S\right),C_{S}\left(Z_{2}\left(S\right)\right)\right\} $
or $E$ is a Pearl. 
\end{prop}

By adding the additional requirement of $S$ to be exceptional, then
Grazian and Parker's classification provides with the following result.
\begin{prop}
\label{prop:maximal-and-exceptional}Let $\F$ be a saturated fusion
system over an exceptional $p$-group $S$ of maximal nihlpotency
class. Assume that $\mathcal{E}\left(\F\right)=\left\{ \gamma_{1}\left(S\right),C_{S}\left(Z_{2}\left(S\right)\right)\right\} $
and that $\gamma_{1}\left(S\right)$ is not abelian. Then $\F$ is
either realizable, contains an abelian Pearl or is an exotic fusion
system over the sylow $7$-subgroup of $G_{2}\left(7\right)$.
\end{prop}

\begin{proof}
Since $S$ is exceptional, then it follows from \cite[Lemma 3.3(v)]{GraPar25}
that $\left|S\right|\ge p^{6}$.

Assume that $p=2$. It follows from \cite[Lemma 6.1]{GraPar25}, that
all the $\F$-essentials are Pearls which contradicts the statement.

Assume that $p=3$. It follows from \cite[Lemma 6.2]{GraPar25}, that
either $\gamma_{1}\left(S\right)\not\in\mathcal{E}\left(\F\right)$
or $\gamma_{1}\left(S\right)$ is abelian. This contradicts the statement.

Assume that $\F=N_{\F}\left(P\right)$ for some $P\in\Fc$. It follows
from \cite[Proposition C]{BCGLO05} that $\F$ is realizable.

The remaining cases are covered in \cite[Theorem B]{GraPar25}
\end{proof}
We can now prove Proposition \ref{prop:sharpness-maximal-nihlpotency-class}.
\begin{proof}[\foreignlanguage{english}{proof of Proposition \ref{prop:sharpness-maximal-nihlpotency-class}}]
If $\left|S\right|\le p^{3}$ the result follows from Corollary \ref{cor:sharpness-groups-of-order-p3}.
Assume that $\left|S\right|\ge p^{4}$.

We know from Proposition \ref{prop:pruning-method} that, by ``pruning''
the Pearls and abelian groups $\mathcal{E}\left(\F\right)$, we obtain
a saturated fusion system. Applying Theorem \hyperref[thm:B]{B} if
necessary, we can assume without loss of generality that there are
no $\F$-essentials of this form. It follows Proposition \ref{prop:classification-essentials-maximal-class}
that $\mathcal{E}\left(\F\right)\subseteq\left\{ \gamma_{1}\left(S\right),C_{S}\left(Z_{2}\left(S\right)\right)\right\} $.

Assume $\left|\mathcal{E}\left(\F\right)\right|\le1$. If $\mathcal{E}\left(\F\right)\not=\emptyset$,
then let $E\in\mathcal{E}\left(\F\right)$. Otherwise define $E:=\gamma_{1}\left(S\right)$.
Since $E\in\left\{ \gamma_{1}\left(S\right),C_{S}\left(Z_{2}\right)\right\} $,
then it is characteristic in $S$. It follows from Alperin's fusion
theorem that $\F=N_{\F}\left(E\right)$. The result follows from Lemma
\ref{lem:sharpness-for-normalizers}.

We can therefore assume that $S$ is exceptional and $\mathcal{E}\left(\F\right)=\left\{ \gamma_{1}\left(S\right),C_{S}\left(Z_{2}\right)\right\} $
with neither $\F$-essential abelian. It follows from Proposition
\ref{prop:maximal-and-exceptional} that $\F$ is either realizable
of a saturated fusion system over a Sylow $7$-subgroup of $G_{2}\left(7\right)$.
In the first case the result follows from \cite[Theorem B]{DiazPark15},
in the second case the result follows from \cite[Theorem A]{GraMar23}.
\end{proof}
\begin{rem}
Let $S$ be a Sylow $7$-subgroup of $G_{2}\left(7\right)$ and let
$\F$ be an exotic fusion system over $\F$ with $\mathcal{E}\left(\F\right)\subseteq\left\{ \gamma_{1}\left(S\right),C_{S}\left(Z_{2}\left(S\right)\right)\right\} $.
If either $\gamma_{1}\left(S\right)$ or $C_{S}\left(Z_{2}\left(S\right)\right)$
does not properly contain any $\F$-centric-radical, then we can apply
Proposition \ref{prop:sharpness-2-essentials} in order to prove Proposition
\ref{prop:sharpness-maximal-nihlpotency-class} independently from
\cite[Theorem A]{GraMar23}. Otherwise a more thorough study of the
$\F$-centric-radical subgroups is needed. Such work falls outside
the scope of this document.
\end{rem}

\subsection{Cohomological sharpness for saturated fusion systems over $p$-groups
of rank $2$ and $p$ odd}

In \cite{DiazRuizViruel07}, Diaz and Viruel provide a classification
of all saturated fusion systems over a $p$-group of rank $2$ and
$p$ odd. Their results can be summarized as follows.
\begin{prop}[{\foreignlanguage{english}{\cite[Theorem 1.1]{DiazRuizViruel07}}}]
\label{prop:classification-rank-2}Let $p$ be an odd prime, let
$S$ be a $p$-group of rank $2$ and let $\F$ be a saturated fusion
system over $S$. Then $\F$ is either realizable or all the $\F$-centric-radical
subgroups of $S$ are either abelian or extraspecial of order $3^{3}$
and exponent $3$.
\end{prop}

Using this classification and Theorem \hyperref[thm:B]{B} we obtain
the following.
\begin{prop}
\label{prop:sharpness-rank-2}Let $p$ be an odd prime, let $S$ be
a $p$-group of rank $2$ and let $\F$ be a saturated fusion system
over $S$. Then $\F$ satisfies cohomological sharpness.

If $\F$ is realizable, then the result follows from \cite[Theorem B]{DiazPark15}.

Assume that $\F$ is exotic and define $\H:=N_{\F}\left(S\right)$.
We know from Lemma \ref{lem:sharpness-for-normalizers}, that $\F$
satisfies sharpness. Proposition \ref{prop:classification-rank-2}
ensures us that every $\F$-essential is either abelian or extraspecial
of order $p^{3}$ and exponent $p$. The result follows from applying
Theorem \hyperref[thm:B]{B} with $\mathcal{A}=\mathcal{E}\left(\F\right)$.
\end{prop}

\subsection{Cohomological sharpness for the Henke-Shpectorov fusion systems}

In their unpublished notes \cite{HenkeShpectorovUnpublished}, Henke
and Shpectorov study saturated fusion systems over a Sylow $p$-subgroup
of $PSp_{4}\left(p^{n}\right)$. Apart from some polynomial fusion
systems, they discovered $2$ exotic fusion systems for $p^{n}=9$.
These fusion systems where confirmed to be saturated by Parker and
Semeraro in \cite[Theorem 5.8]{ParkerSemeraro21} and satisfy the
following properties.
\begin{prop}[{\foreignlanguage{english}{\cite[Theorem 5.8]{ParkerSemeraro21}}}]
\label{prop:classification-henke-shpectorov}Let $S$ be a Sylow
$3$-subgroup of $PSp_{4}\left(9\right)$ and let $A:=J\left(S\right)$
be its Thompson group. Then there exists an elementary abelian subgroup
$E\le S$ of order $81$ such that $S=AE$ and a two exotic saturated
fusion systems on $S$ with $\mathcal{E}\left(\F\right)=A\cup E^{S}$.
\end{prop}

Using this description we obtain the following.
\begin{prop}
\label{prop:sharpness-Henke-Shpectorov}Let $\F$ be one of the fusion
systems of Proposition \ref{prop:classification-henke-shpectorov}
and let $M^{*}:\OF^{\operatorname{op}}\to\Fp\mod{}$ be the contravariant
part of any Mackey functor over $\F$ with coefficients in $\Fp[3]$.
Then $\limn_{\OFc}\left(M^{*}\downarrow_{\OFc}^{\OF}\right)=0$ for
every $n\ge2$. In particular, $\F$ satisfies cohomological sharpness
(see \cite[Lemma 3.6]{Pra26}).
\end{prop}

\begin{proof}
Since $A$ is $S$-characteristic, then $\Aut_{N_{\F}\left(A\right)}\left(S\right)=\Aut_{\F}\left(S\right)$.
From Alperin-Goldschmidt's fusion theorem we deduce that $\F=\left\langle N_{\F}\left(A\right),\Aut_{\F}\left(E\right)\right\rangle _{S}$.
The result follows from applying Proposition \ref{prop:sharpness-2-essentials}
with $P=A$ and $Q=E$.
\end{proof}
\begin{rem}
All other exotic fusion systems studied by Henke and Shpectorov in
\cite{HenkeShpectorovUnpublished} are polynomial in the sense of
Definition \ref{def:polynomial-fusion-system} (see \cite[Theorem 1.1]{GPSB26}).
Therefore, Proposition \ref{prop:sharpness-polynomial} ensure that
Proposition \ref{prop:sharpness-Henke-Shpectorov} holds for all Henke-Shpectorov
fusion systems.
\end{rem}

\subsection{Cohomological sharpness for $15$ van Beek fusion systems}\label{subsec:Sharpness-some-van-Beek}

In \cite{VanBeek25}, van Beek finds $16$ exotic simple fusion systems
related to either $3$-local structure of the Thompson group $F_{3}$
or the $5$-local structure of the Monster group $M$. During this
section we study $15$ of these exotic fusion systems. The remaining
fusion system is studied in Subsection \ref{subsec:last-van-beek-fusion-system}.

\subsubsection{Fusions systems related to $3$-local structure of $F_{3}$.}

Let $F_{3}$ be the usual Thompson group and let $S\in\Syl_{3}\left(F_{3}\right)$.
As explained in \cite[Section 5]{VanBeek25}, we can take the following
three maximal $3$-local subgroups of the Thompson group $F_{3}$
\begin{align*}
M_{1} & \cong C_{3}^{2}.C_{3}^{3}.C_{3}^{2}.C_{3}^{2}\rtimes\GL_{2}\left(3\right),\\
M_{2} & \cong C_{3}.C_{3}^{2}.C_{3}.C_{3}^{2}.C_{3}.C_{3}^{2}\rtimes\GL_{2}\left(3\right),\\
M_{3} & \cong C_{3}^{5}\rtimes\SL_{2}\left(9\right).C_{2},
\end{align*}
where $X.Y$ denotes an extension of $Y$ by $X$. By taking $F_{3}$-conjugates
if necessary, we can assume without loss of generality that $S\cap M_{i}\in\Syl_{3}\left(M_{i}\right)$
for $i=1,2,3$. Adopting Notation \ref{nota:central-series}, we can
further define
\begin{align*}
E_{1} & :=O_{3}\left(M_{1}\right)=C_{S}\left(Z_{2}\left(S\right)\right),\\
E_{2} & :=O_{3}\left(M_{2}\right)=C_{S}\left(Z_{3}\left(S\right)/Z\left(S\right)\right),\\
E_{3} & :=O_{3}\left(M_{3}\right)\cong C_{3}^{5}.
\end{align*}
It follows from the above that $E_{1}$ and $E_{2}$ are characteristic
in $S$. With this notation, van Beek proves the following.
\begin{prop}
\label{prop:classification-van-beek-F3}Let $\F$ be a saturated fusion
system over $S$ such that $O_{3}\left(\F\right)=\left\{ 1\right\} $.
Then $\F$ is isomorphic to one of three fusion systems which we denote
by $\G$, $\H$ and $\D$. Moreover $\G=\F_{S}\left(F_{3}\right)$
while $\H$ and $\D$ are exotic and satisfy
\begin{align*}
\Fcr[H] & =\left\{ S,E_{1},E_{2}\right\} , & \Fcr[D] & =\left\{ S,E_{1}\right\} \cup E_{3}^{S}.
\end{align*}
\end{prop}

\begin{proof}
The first part of the statement is given in \cite[Theorem B]{VanBeek25}.
The fusion systems $\H$ and $\D$ and their centric-radicals are
given in \cite[Propositions 5.4 and 5.7]{VanBeek25} respectively.
They are proven to be exotic in Propositions \cite[Propositions 5.6 and 5.9]{VanBeek25}.
\end{proof}
Using this classification we can prove the following.
\begin{prop}
\label{prop:sharpness-van-beek-F3}Let $\F$ be a saturated fusion
system over a Sylow $3$ subgroup of the Thompson group $F_{3}$ and
let $M^{*}:\OF^{\operatorname{op}}\to\Fp[3]\mod{}$ be the contravariant
part of any Mackey functor over $\F$ with coefficients in $\Fp[3]$.
If $O_{3}\left(\F\right)=\left\{ 1\right\} $, then $\limn_{\OFc}\left(M^{*}\downarrow_{\OFc}^{\OF}\right)=0$
for every $n\ge2$. In particular, $\F$ satisfies cohomological sharpness
(see \cite[Lemma 3.6]{Pra26}).
\end{prop}

\begin{proof}
For $\F=\G$ the result follows from \cite[Theorem B]{DiazPark15}.
For $\F=\H$ (resp. $\F=\D$) let $P=E_{1}$ and $Q=E_{2}$ (resp.
$Q=E_{3}$). Proposition \ref{prop:classification-van-beek-F3} ensures
that $Q$ is minimal $\F$-centric-radical. Since $P$ is characteristic
in $S$, then $\Aut_{\F}\left(S\right)=\Aut_{N_{\F}\left(P\right)}\left(S\right)$.
It follows from Alperin-Goldschmidt's fusion theorem and Proposition
\ref{prop:classification-van-beek-F3} that $\F=\left\langle N_{\F}\left(P\right),\Aut_{\F}\left(Q\right)\right\rangle $.
The result follows from applying Proposition \ref{prop:sharpness-2-essentials}.
\end{proof}
Van Beek also describes the fusion systems over $E_{1}$.
\begin{prop}
\label{prop:classification-van-beek-E1-F3}Let $\F$ be a saturated
fusion system over $E_{1}$. Then either $\F=N_{\F}\left(E_{1}\right)$
or $\F$ is isomorphic to one of two exotic fusion systems and $\Fcr=\left\{ E_{1}\right\} \cup E_{3}^{E_{1}}$.
\end{prop}

\begin{proof}
The first part of the statement is given in \cite[Theorem D]{VanBeek25}.
The two exotic fusion systems are described in \cite[Proposition 5.19]{VanBeek25}
and their centric-radicals are given in \cite[Proposition 5.20]{VanBeek25}.
\end{proof}
As a consequence of this description we obtain the following.
\begin{prop}
\label{prop:sharpness-van-beek-E1-F3}Let $S$ be a Sylow $3$-subgroup
of the Thompson group $F_{3}$, let $\F$ be a saturated fusion system
over $C_{S}\left(Z_{2}\left(S\right)\right)$ and let $M^{*}:\OF^{\operatorname{op}}\to\Fp[3]\mod{}$
be the contravariant part of any Mackey functor over $\F$ with coefficients
in $\Fp[3]$. Then $\limn_{\OFc}\left(M^{*}\downarrow_{\OFc}^{\OF}\right)=0$
for every $n\ge2$. In particular, $\F$ satisfies cohomological sharpness
(see \cite[Lemma 3.6]{Pra26}).
\end{prop}

\begin{proof}
For $\F=N_{\F}\left(E_{1}\right)$, the result follows from Lemma
\ref{lem:sharpness-for-normalizers}.

Otherwise let $P=E_{1}$ and $Q=E_{3}$. Proposition \ref{prop:classification-van-beek-E1-F3}
ensures that $Q$ is the unique (up to $S$-conjugacy) proper $\F$-centric-radical.
It follows from Alperin-Goldschmidt's fusion theorem that $\F=\left\langle N_{\F}\left(P\right),\Aut_{\F}\left(Q\right)\right\rangle _{S}$.
The result follows from applying Proposition \ref{prop:sharpness-2-essentials}.
\end{proof}

\subsubsection{Fusions systems related to the $5$-local structure of $M$.}\label{subsubsec:fusion-systems-on-monster}

Let $M$ be the monster group and let $S\in\Syl_{5}\left(M\right)$.
As explained in \cite[Section 6]{VanBeek25}, we can take the following
two maximal $5$-local subgroups of $M$
\begin{align*}
M_{1} & \cong C_{5}^{2}.C_{5}^{2}.C_{5}^{4}\rtimes\left(S_{3}\times\GL_{2}\left(5\right)\right),\\
M_{3} & \cong C_{5}^{4}\rtimes\left(C_{3}\times\SL_{2}\left(25\right)\right).C_{2}.
\end{align*}
By taking $M$-conjugates if necessary, we can assume without loss
of generality that $S\cap M_{i}\in\Syl_{5}\left(M_{i}\right)$ for
every $i=1,3$. Adopting Notation \ref{nota:central-series}, we can
further define
\begin{align*}
E_{1} & :=O_{3}\left(M_{1}\right)=C_{S}\left(Z_{2}\left(S\right)\right), & E_{3} & :=O_{3}\left(M_{3}\right)\cong C_{5}^{4}.
\end{align*}
Notice that $E_{1}$ is characteristic in $S$. With this notation
van Beek proves the following.
\begin{prop}
\label{prop:classification-van-beek-M}Let $\F$ be a saturated fusion
system over $S$ such that $O_{5}\left(\F\right)=\left\{ 1\right\} $.
Then $\F$ is isomorphic to either the fusion system $\G=\F_{S}\left(M\right)$,
the fusion system $\H$ defined in Subsection \ref{subsec:last-van-beek-fusion-system}
or to one of other $2$ exotic fusion systems which satisfy $\Fcr=\left\{ S,E_{1}\right\} \cup E_{3}^{\F}$.
Here $E_{3}^{\F}$ denotes the set of subgroups of $S$ that are $\F$-conjugate
to $E_{3}$.
\end{prop}

\begin{proof}
The first part of the statement is given in \cite[Theorem C]{VanBeek25}.
The fusion systems $\H$ is given in \cite[Proposition 6.10]{VanBeek25}
while the two remaining fusion systems are given in \cite[Proposition 6.13]{VanBeek25}.
Let $\F$ be one of these remaining fusion systems. We know from \cite[Proposition 6.14]{VanBeek25},
that $\Fcr=\left\{ S,E_{1}\right\} \cup E_{3}^{\G}$. If no $\G$-conjugate
of $E_{3}$ was $\F$-essential, then we would have $\F=N_{\F}\left(E_{1}\right)$.
In particular $\F$ would be realizable (see \cite[Proposition C]{BCGLO05}).
From \cite[Proposition 6.15]{VanBeek25} we know that that is not
the case. We deduce that at least one $\G$-conjugate of $E_{3}$
is $\F$-essential. We conclude from \cite[Proposition 6.9(ii)]{VanBeek25}
that $E_{3}^{\G}=E_{3}^{\F}$. The result follows.
\end{proof}
Using this classification we can prove the following.
\begin{prop}
\label{prop:sharpness-van-beek-M}Let $\F\not\cong\H$ be a saturated
fusion system over a Sylow $5$-subgroup of the Monster and let $M^{*}:\OF^{\operatorname{op}}\to\Fp[5]\mod{}$
be the contravariant part of any Mackey functor over $\F$ with coefficients
in $\Fp[5]$. If $O_{5}\left(\F\right)=\left\{ 1\right\} $, then
$\limn_{\OFc}\left(M^{*}\downarrow_{\OFc}^{\OF}\right)=0$ for every
$n\ge2$. In particular, $\F$ satisfies cohomological sharpness (see
\cite[Lemma 3.6]{Pra26}).
\end{prop}

\begin{proof}
For $\F=\G$ the result follows from \cite[Theorem B]{DiazPark15}.
Otherwise, since $P$ is characteristic in $S$, then $\Aut_{\F}\left(S\right)=\Aut_{N_{\F}\left(E_{1}\right)}\left(S\right)$.
It follows from Alperin-Goldschmidt's fusion theorem and Proposition
\ref{prop:classification-van-beek-M} that $\F=\left\langle N_{\F}\left(E_{1}\right),\Aut_{\F}\left(E_{3}\right)\right\rangle _{S}$.
The result follows from applying Proposition \ref{prop:sharpness-2-essentials}
with $P=E_{1}$ and $Q=E_{3}$.
\end{proof}
Van Beek also provides a description of the fusion systems over $E_{1}$.
\begin{prop}
\label{prop:classification-van-beek-E1-M}Let $\F$ be a saturated
fusion system over $E_{1}$. Then either $\F=N_{\F}\left(E_{1}\right)$
or $\F$ is isomorphic to one of nine exotic fusion systems and $\Fcr\subseteq\left\{ E_{1}\right\} \cup E_{3}^{\G}$.
\end{prop}

\begin{proof}
The first part of the statement is given in \cite[Theorem E]{VanBeek25}.
´The first of the nine exotic fusion systems (denoted $\D^{*}$) is
given in \cite[Proposition 6.26]{VanBeek25}. Then \cite[Proposition 6.29]{VanBeek25}
tells us that, up to isomorphism, there exist $5$ fusion subsystems
of $\D^{*}$ of index prime to $5$. Moreover $\Fcr=\left\{ E_{1}\right\} \cup E_{3}^{\D^{*}}$
for all such fusion systems. \cite[Proposition 6.30]{VanBeek25} ensures
that they are all exotic.

In \cite[Page 48]{VanBeek25} the saturated fusion subsystem $\mathcal{L_{P}}\subseteq\D^{*}$
is defined by ``pruning'' (see Proposition \ref{prop:pruning-method})
one of the fusion subsystems of $\D^{*}$ of index prime to $5$.
Finally \cite[Proposition 6.32]{VanBeek25}, tells us that, up to
isomorphism, there exist $4$ fusion subsystems of $\mathcal{L_{P}}$
of index prime to $5$. Moreover all these fusion systems are exotic
and satisfy $\Fcr=\left\{ E_{1}\right\} \cup E_{3}^{O^{5'}\left(\D^{*}\right)}$.
This concludes the proof.
\end{proof}
As a consequence of this description we obtain the following.
\begin{prop}
\label{prop:sharpness-van-beek-E1-M}Let $S$ be a Sylow $5$-subgroup
of the Monster group $M$ and let $\F$ be a saturated fusion system
over $C_{S}\left(Z_{2}\left(S\right)\right)$. Then $\F$ satisfies
cohomological sharpness.
\end{prop}

\begin{proof}
For $\F=N_{\F}\left(E_{1}\right)$, the result follows from Lemma
\ref{lem:sharpness-for-normalizers}.

From Proposition \ref{prop:classification-van-beek-M} and Alperin-Goldschmidt's
fusion theorem we know that there exist $\mathcal{A}\subseteq E_{3}^{\G}$
such that $\F=\left\langle N_{\F}\left(E_{1}\right),\Aut_{\F}\left(A\right)\mid A\in\mathcal{A}\right\rangle _{E_{1}}$.
Since $E_{3}$ is abelian, then the result follows from applying Theorem
\hyperref[thm:B]{B}.
\end{proof}

\section{Cohomological sharpness via matching normalizers}\label{sec:sharpness_via_normalizers}

As shown during Section \ref{sec:sharpness_via_pruning}, Theorem
\hyperref[thm:A]{A} is an efficient tool towards studying higher
limits in whenever $\CFp$ vanishes. However, this functor does not
always vanish. The Benson-Solomon fusion systems are a prime example
of this (see \cite[Subsection 4.3]{Pra26}). Theorem \hyperref[thm:C]{C}
provides us with methods towards studying higher limits in this cases.
\begin{notation}
\label{nota:with-normalizers}Notation \ref{nota:C-included-in-notation}
is adopted, $\F$, $\F_{\boldsymbol{1}}$, $\F_{\boldsymbol{2}}$
and $\F_{\boldsymbol{e}}$ are saturated and $Q\le S'$ is a fully
$\F$-normalized element in $\mathcal{C}$. Moreover, for every $x\in\Ob\left(\T\right)$,
we define the fusion system $\G_{x}:=N_{\F_{x}}\left(Q\right)$ over
$N_{S_{x}}\left(Q\right)$. We also define the fusion system $\G:=\left\langle \G_{\boldsymbol{1}},\G_{\boldsymbol{2}}\right\rangle _{N_{S}\left(Q\right)}$,
the set $\Xi:=\left\{ \G_{\boldsymbol{1}},\G_{\boldsymbol{2}},\G_{\boldsymbol{e}}\right\} $
and let $CY_{0}$ and $CY_{1}$ be as in Equation (\ref{eq:CY0-CY1}).
\end{notation}

\begin{proof}[\foreignlanguage{english}{Proof of Theorem C}]
The second part of the statement follows from the first and \cite[Lemma 3.6]{Pra26}.

Throughout the rest of this proof we adopt Notation \ref{nota:with-normalizers}.
This is consistent with the notation in the statement.

Since $Q$ is fully $\F$-normalized and $\F$-centric then it is
fully $\H$-normalized and $\H$-centric for every $\H\in\Lambda\cup\left\{ \F\right\} $.
Since each such $\H$ is saturated, then it follows from Lemma \ref{lem:sharpness-for-normalizers}
that every $\mathcal{E}\in\Xi\cup\left\{ \G\right\} $ satisfies sharpness.

It follows from Lemma \ref{lem:centric-radical-contained}, that $\Fcr[E]\subseteq\mathcal{C}$.
We can therefore apply Theorem \hyperref[thm:A]{A}(\ref{enu:thm-a-iso})
with $\Xi$ in place of $\Lambda$. It follows from Shapiro's isomorphism
(see \cite[Proposition 4.8]{Yal22}) that, for every $n\ge3$,
\[
\Ext_{\OFC}^{n}\left(\CFp,M^{*}\downarrow_{\OFC}^{\OF}\right)\cong\Ext_{\OFC}^{n}\left(\CFp[\Xi],M^{*}\downarrow_{\OFC[\G]}^{\OF}\right)\cong\limn_{\G}\left(M^{*}\downarrow_{\OFc[\G]}^{\OF}\right)=0.
\]
It follows from Theorem \hyperref[thm:A]{A}(\ref{enu:thm-a-iso})
(applied with $\Lambda$) that $\limn_{\F}\left(M\right)=0$ for every
$n\ge3$.

Apply Theorem \hyperref[thm:A]{A}(\ref{enu:thm-A-exact-seq}) with
$\Xi$ in place of $\Lambda$. Since $\limn[2]_{\G}\left(M^{*}\downarrow_{\OFc[\G]}^{\OF}\right)=0$,
then we deduce that the morphism
\[
\Psi:\Nat_{\OFC}\left(CY_{1},M\downarrow_{\OFC[\G]}^{\OFc}\right)\twoheadrightarrow\Nat_{\OFC}\left(\CFp[\Xi],M\downarrow_{\OFC[\G]}^{\OFc}\right),
\]
arising from the natural inclusion $\CFp[\Xi]\hookrightarrow CY_{1}$
is an epimorphism.

Let $\xi:\Nat_{\OFC}\left(CY_{1}\uparrow_{\OFC[\G]}^{\OFC},M\downarrow_{\OFC}^{\OFc}\right)\twoheadrightarrow\Nat_{\OFC}\left(CX_{1},M\downarrow_{\OFC}^{\OFc}\right)$
be such that $\xi\pi_{1}^{*}=\Id$. From adjointness of the pair $\left(\downarrow_{\OFC[\G]}^{\OFC},\uparrow_{\OFC[\G]}^{\OFC}\right)$,
we deduce that the following diagram commutes
\[
\xymatrix{\Nat_{\OFC}\left(CX_{1},M\downarrow_{\OFC}^{\OFc}\right)\ar@{->}[d]^{\varUpsilon}\ar@/^{2pc}/@{^{(}->}[r]^{\pi_{1}^{*}} & \Nat_{\OFC}\left(CY_{1}\uparrow_{\OFC[\G]}^{\OFC},M\downarrow_{\OFC}^{\OFc}\right)\ar@/^{2pc}/@{->>}[l]_{\xi}\ar@{->>}[d]^{\tilde{\Psi}}\\
\Nat_{\OFC}\left(\CFp,M\downarrow_{\OFC}^{\OFc}\right)\ar@{^{(}->>}[r]_{\Gamma^{*}} & \Nat_{\OFC}\left(\CFp[\Xi]\uparrow_{\OFC[\G]}^{\OFC},M\downarrow_{\OFC}^{\OFc}\right)
}
,
\]
where $\tilde{\Psi}$ is the image of $\Psi$ via the adjointness
isomorphism and $\Gamma^{*}$ results from applying the functor $\Nat_{\OFC}\left(-,M\downarrow_{\OFC}^{\OFc}\right)$
to the isomorphism $\Gamma$ of Equation (\ref{eq:def-gamma}). Since
both $\tilde{h}$ and $\left(\Gamma^{*}\right)^{-1}$ are epimorphisms,
then it follows that $\varUpsilon\xi=\left(\Gamma^{*}\right)^{-1}\tilde{\Psi}$
is also an epimorphism. In particular $\varUpsilon$ is an epimorphism.
It follows from Theorem \hyperref[thm:A]{A}(\ref{enu:thm-A-exact-seq})
(applied with $\Lambda$) that $\limn[2]_{\F}\left(M\right)=0$. This
concludes the proof.
\end{proof}
The following provides sufficient conditions for the monomorphism
$\pi_{1}^{*}$ of Theorem \hyperref[thm:C]{C} to split.
\begin{prop}
\label{prop:sufficient-condition-for-splitting.}With notation as
in Theorem \hyperref[thm:C]{C}. Assume that $M=H^{j}\left(-;\Fp\right)$
for some $j\ge0$, that the fusion system $\F_{\boldsymbol{e}}$ is
realizable and that $Q\trianglelefteq S'$. Then $\pi_{1}^{*}$ splits.

Since $Q\in\mathcal{C}$, then $\Fcr[G]_{\boldsymbol{e}}\subseteq\mathcal{C}$
(see Lemma \ref{lem:centric-radical-contained}). Let $\H\in\left\{ \G_{\boldsymbol{e}},\F_{\boldsymbol{e}}\right\} $.
Since $\H$ is saturated and $\Fcr[H]\subseteq\mathcal{C}$, then
Alperin-Goldschmidt's fusion theorem gives us a natural isomorphism
\[
M^{\H}:=\lim_{\OFC[\H]}\left(H^{j}\left(-;\Fp\right)\right)\cong\lim_{\OF[\H]}\left(H^{j}\left(-;\Fp\right)\right).
\]
Recall the natural isomorphism $M^{\H}\cong\Nat_{\OFC[\H]}\left(\underline{\Fp},H^{j}\left(-;\Fp\right)\right)$.
From adjointness of the pair $\left(\uparrow_{\OFC[\H]}^{\OFc},\downarrow_{\OFC[\H]}^{\OFc}\right)$,
we deduce that $\pi_{1}^{*}$ splits if and only if the universal
map $M^{\F_{\boldsymbol{e}}}\to M^{\G_{\boldsymbol{e}}}$ splits.

Let $G$ be a finite group such that $S'\in\Syl_{p}\left(G\right)$
and that $\F_{\boldsymbol{e}}=\F_{S'}\left(G\right)$. Then $\G_{\boldsymbol{e}}$
is realized by $N_{G}\left(Q\right)$. We know from the stable elements
theorem, that there exist natural isomorphism $H^{j}\left(G;\Fp\right)\cong M^{\F_{\boldsymbol{e}}}$
and $H^{j}\left(N_{G}\left(Q\right);\Fp\right)\cong M^{\G_{\boldsymbol{e}}}$.
We deduce that $\pi_{1}^{*}$ splits if and only if the morphism $\operatorname{Res}:H^{j}\left(G;\Fp\right)\to H^{j}\left(N_{G}\left(Q\right);\Fp\right)$
deriving from the natural inclusion $N_{G}\left(Q\right)\subseteq G$
splits. Let $\operatorname{tr}:H^{j}\left(N_{G}\left(Q\right);\Fp\right)\to H^{j}\left(G;\Fp\right)$
denote the transfer map and define $n:=\left[G:N_{G}\left(Q\right)\right]$.
It is well known that the composition $\operatorname{tr}\operatorname{Res}$
equals multiplication by $n$. Since $p\not\mid n$ whenever $Q\trianglelefteq S'$,
then this composition is an isomorphism. The result follows.
\end{prop}

\begin{rem}
The proof of Proposition \ref{prop:sufficient-condition-for-splitting.}
depends heavily on the assumption $M=H^{j}\left(-;\Fp\right)$. The
same result may not be valid if $M$ cannot be lifted to a global
Makey functor or if $\lim_{\OFc[\FSG]}\left(M^{*}\right)\not\cong M^{*}\left(G\right)$
for some realizable fusion subsystem.
\end{rem}

The following provides sufficient conditions for the morphism $\Gamma$
of Theorem \hyperref[thm:C]{C} to be an isomorphism.
\begin{lem}
\label{lem:Description-induction}Adopt Notation \ref{nota:with-normalizers}
and let $P\in\mathcal{C}$. The following holds for every commutative
ring $\R$
\[
\CFR[\Xi]\uparrow_{\OFC[\G]}^{\OFC}\left(P\right)\cong\bigoplus_{\left[\varphi\right]_{\G}\in\Rep_{\F}\left(P,\G\right)}\CFR[\Xi]\left(\varphi\left(P\right)\right).
\]
\end{lem}

\begin{proof}
This is immediate from the description of induction of functors given,
for example, in \cite[Proposition 4.5]{Yal22}.
\end{proof}
\begin{cor}
\label{cor:relation-induction-no-induction}Adopt Notation \ref{nota:with-normalizers}
and let $P\in\mathcal{C}$. Assume that $\Hom_{\G}\left(P,N_{S}\left(Q\right)\right)=\Hom_{\F}\left(P,N_{S}\left(Q\right)\right)$.
The following isomorphism holds for every commutative ring $\R$
\[
\CFR[\Xi]\uparrow_{\OFC[\G]}^{\OFC}\left(P\right)\cong\CFR[\Xi]\left(P\right).
\]
Moreover, if $P=Q$, then the morphism $\Gamma$ of Equation (\ref{eq:def-gamma})
leads to an isomorphism
\[
\CFR[\Xi]\uparrow_{\OFC[\G]}^{\OFC}\left(Q\right)\cong\CFR\left(Q\right).
\]
\end{cor}

\begin{proof}
Under the given assumptions we have that $\Rep_{\F}\left(P,\G\right)=\left\{ \left[\Id_{P}\right]_{\G}\right\} $.
The first isomorphism follows from Lemma \ref{lem:Description-induction}.
If $P=Q$ then, the inclusion $\G\subseteq N_{\F}\left(Q\right)$
tells us that $Q$ is weakly $\F$-closed. It follows that, for every
$x\in\Ob\left(\T\right)$, then 
\[
\Rep_{\G}\left(Q,\G_{x}\right)=\Aut_{\G}\left(Q\right)/\Aut_{\G_{x}}\left(Q\right)=\Aut_{\F}\left(Q\right)/\Aut_{\F_{x}}\left(Q\right)\cong\Rep_{\F}\left(Q,\F_{x}\right).
\]
We conclude that $\Rep_{\G}\left(Q,\Xi\right)\cong\Rep_{\F}\left(Q,\Lambda\right)$.
The result follows.
\end{proof}
As an application of the above tools results we prove cohomological
sharpness for the van Beek fusion system defined in \cite[Proposition 6.10]{VanBeek25}.
This completes the proof of Theorem \hyperref[thm:D]{D}.

\subsection{Cohomological sharpness for the last van Beek fusion system}\label{subsec:last-van-beek-fusion-system}

Let us start by recalling and complementing some of the notation introduced
in Subsubsection \ref{subsubsec:fusion-systems-on-monster}.

Let $M$ be the monster group and let $S\in\Syl_{5}\left(M\right)$
a subgroup of order $5^{9}$. As explained in \cite[Section 6]{VanBeek25},
there exist three maximal $5$-local subgroups $M_{1},M_{2},M_{4}\le M$
satisfying
\begin{align*}
M_{1} & \cong C_{5}^{2}.C_{5}^{2}.C_{5}^{4}\rtimes\left(S_{3}\times\GL_{2}\left(5\right)\right),\\
M_{2} & \cong5_{+}^{1+6}\rtimes C_{4}.J_{2}.C_{2},\\
M_{4} & \cong C_{5}^{3}.C_{5}^{3}.\left(C_{2}\times\operatorname{PSL}_{3}\left(5\right)\right),
\end{align*}
where $X.Y$ denotes an extension of $Y$ by $X$. By taking $M$-conjugates
if necessary, we can assume without loss of generality that $S\cap M_{i}\in\Syl_{5}\left(M_{i}\right)$.
Define the subgroups of $S$
\begin{align*}
E_{1} & :=O_{5}\left(M_{1}\right)\cong C_{5}^{2}.C_{5}^{2}.C_{5}^{4}, & \boldsymbol{Q} & :=O_{5}\left(M_{2}\right)\cong5_{+}^{1+6}, & \boldsymbol{R} & :=O_{5}\left(M_{4}\right)\cong C_{5}^{3}.C_{5}^{3}.
\end{align*}
The group $E_{1}$ can alternatively be defined as $E_{1}=C_{S}\left(Z_{2}\left(S\right)\right)$
and, in particular, it is $S$-characteristic. Since $E_{1}$ is maximal
in $S$, then it follows from Alperin's fusion theorem that $E_{1}$
is weakly closed with respect to any saturated fusion system over
$S$. On the other hand, the group $\boldsymbol{Q}$ is the unique
extraspecial subgroup of $S$ of order $5^{7}$. In particular, $\boldsymbol{Q}$
is also weakly closed with respect to any fusion system over $S$.

Take the subgroup $X\triangleleft M_{2}$ with $X\cong5_{+}^{1+6}\rtimes C_{4}$
and let $X\trianglelefteq H\le M_{2}$ be the maximal subgroup such
that $H/X\cong\left(A_{5}\times D_{10}\right).C_{2}$. Define $E_{2}\le O_{5}\left(H\right)$
to be the largest subgroup such that
\[
N_{M}\left(E_{2}\right)=N_{M_{2}}\left(E_{2}\right)\cong5_{+}^{1+6}.C_{5}\rtimes\left(C_{2}\times\GL_{2}\left(5\right)\right).
\]
With this setup, it is possible to define the fusion system $\H$
over $S$ as 
\[
\H:=\left\langle \Aut_{M}\left(E_{1}\right),\Aut_{M}\left(E_{2}\right),\Aut_{M}\left(S\right)\right\rangle _{S}.
\]

The $\H$-centric-radical subgroups are then given by the following.
\begin{prop}[{\foreignlanguage{english}{\cite[Proposition 6.10]{VanBeek25}}}]
\label{prop:centric-radicals-last-van-beek}For every $P\le S$ let
$P^{\H}$ denote the set of all $\H$-conjugates of $P$. Then $\H$
is a saturated fusion system with $\mathcal{E}\left(\H\right)=\left\{ E_{1}\right\} \cup E_{2}^{\H}$
and $\Fcr[H]=\left\{ S,E_{1},\boldsymbol{Q}\right\} \cup E_{2}^{\H}\cup\boldsymbol{R}^{\H}$.
\end{prop}

As explained in \cite[Page 35]{VanBeek25}, the $\H$-automorphisms
of $S$ conjugate both $E_{2}$ and $\boldsymbol{R}$ to exactly three
subgroups of $S$. Since $E_{2}$ is maximal in $S$, then these are
all the $\H$-conjugates of $E_{2}$ and we can write 
\[
E_{2}^{\H}=\left\{ E_{2}^{1}:=E_{2},E_{2}^{2},E_{2}^{3}\right\} .
\]
We know from \cite[Appendix A]{VanBeek2025WithCode}, that $\Phi\left(\boldsymbol{R}\right)=\left[E_{2},\Phi\left(E_{2}\right)\right]$
is characteristic in $E_{2}$. We deduce that $\boldsymbol{R}=C_{E_{2}}\left(\Phi\left(\boldsymbol{R}\right)\right)$
is also characteristic in $E_{2}$. In particular, $\Aut_{\H}\left(E_{2}\right)$
normalizes $\boldsymbol{R}$. Let $\boldsymbol{R}'$ be an $\Aut_{\H}\left(S\right)$-conjugate
contained in $E_{2}$. If $\boldsymbol{R}\not=\boldsymbol{R}'$, then
the previous discussion tells us that $\boldsymbol{R}'\le E_{2}^{i}$
for some $i=2,3$. Since $\boldsymbol{Q}=E_{2}\cap E_{2}^{i}$, then
it follows that $\boldsymbol{R}'\le\boldsymbol{Q}$. From the isomorphism
classes of $\boldsymbol{Q}$ and $\boldsymbol{R}$ we deduce that
this is impossible. We conclude that $\boldsymbol{R}$ is the only
$\Aut_{\H}\left(S\right)$-conjugate of $\boldsymbol{R}$ contained
in $E_{2}$. An analogous result holds for the conjugates contained
in $E_{2}^{i}$. From \cite[Lemma 6.6]{VanBeek25}, we know that $\mathcal{O}^{5'}\left(\Aut_{\H}\left(E_{1}\right)\right)$
normalizes every $\Aut_{\H}\left(S\right)$-conjugate of $\boldsymbol{R}$
and that $\mathcal{O}^{5'}\left(\Out_{\H}\left(E_{1}\right)\right)\cong\SL_{2}\left(5\right)$.
From \cite[Table 4]{VanBeek25}, we know that 
\begin{align*}
\Out_{\H}\left(E_{1}\right) & \cong S_{3}\times\GL_{2}\left(5\right), & \Out_{\H}\left(S\right) & \cong S_{3}\times C_{4}^{2}.
\end{align*}
We conclude that the following holds
\begin{lem}
\label{lem:decomposition-aut-in-E1-last-van-beek.}Every automorphism
$\varphi$ in $\Aut_{\H}\left(E_{1}\right)$ can be written as $\varphi=\psi\theta$
with $\theta\in\mathcal{O}^{5'}\left(\Aut_{\H}\left(E_{1}\right)\right)\le N_{\Aut_{\H}\left(E_{1}\right)}\left(\boldsymbol{R}\right)$
and $\psi$ the restriction of an automorphism $\hat{\psi}\in\Aut_{\H}\left(S\right)$.
\end{lem}

We conclude that the $\Aut_{\H}\left(E_{1}\right)$-conjugates of
$\boldsymbol{R}$ coincide with its $\Aut_{\H}\left(S\right)$-conjugates
and, therefore with its $\H$-conjugates. We can therefore write 
\[
\boldsymbol{R}^{\H}=\left\{ \boldsymbol{R}^{1}:=\boldsymbol{R},\boldsymbol{R}^{2},\boldsymbol{R}^{3}\right\} .
\]
The previous discussion is summarized in the following diagram depicting
all the $\H$-centric-radicals ordered by inclusion. Horizontal arrows
denote conjugation in $\H$.
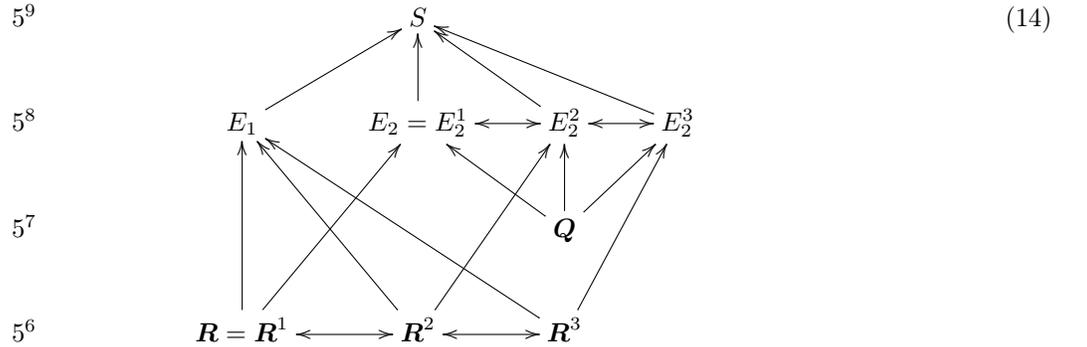
\begin{equation}
\xymatrix{5^{9} &  &  & S\\
5^{8} &  & E_{1}\ar@{->}[ru] & E_{2}=E_{2}^{1}\ar@{<->}[r]\ar@{->}[u] & E_{2}^{2}\ar@{<->}[r]\ar@{->}[lu] & E_{2}^{3}\ar@{->}[llu]\\
5^{7} &  &  &  & \boldsymbol{Q}\ar@{->}[lu]\ar@{->}[u]\ar@{->}[ru]\\
5^{6} &  & \boldsymbol{R}=\boldsymbol{R}^{1}\ar@{<->}[r]\ar@{->}[ruu]\ar@{->}[uu] & \boldsymbol{R}^{2}\ar@{<->}[r]\ar@{->}[ruu]\ar@{->}[luu] & \boldsymbol{R}^{3}\ar@{->}[ruu]\ar@{->}[lluu]
}
\label{eq:hasse-diag-last-van-beek}
\end{equation}
We can now finally prove that cohomological sharpness holds for the
last remaining van Beek fusion system.
\begin{prop}
\label{prop:Sharpness-last-Van-Beek}Let $\H$ be the van Beek fusion
system of Proposition \ref{prop:centric-radicals-last-van-beek}.
Then $\H$ satisfies cohomological sharpness.
\end{prop}

\begin{proof}
Let $\mathcal{C}$ be the smallest family of subgroups of $S$ containing
$\Fcr[H]$ and closed under $\H$-conjugacy and taking overgroups.
Define the fusion systems 
\begin{align*}
\F_{\boldsymbol{1}} & :=N_{\H}\left(E_{1}\right), & \F_{\boldsymbol{2}} & :=N_{\H}\left(\boldsymbol{R}\right), & \F_{\boldsymbol{e}} & :=N_{\F_{\boldsymbol{1}}}\left(\boldsymbol{R}\right),\\
\G_{\boldsymbol{1}} & :=N_{\F_{\boldsymbol{1}}}\left(\boldsymbol{Q}\right), & \G_{\boldsymbol{2}} & :=N_{\F_{\boldsymbol{2}}}\left(\boldsymbol{Q}\right), & \G_{\boldsymbol{e}} & :=N_{\F_{\boldsymbol{e}}}\left(\boldsymbol{Q}\right).
\end{align*}
This is consistent with Notation \ref{nota:with-normalizers} taken
with $Q=\boldsymbol{Q}$. We adopt this notation for the reminder
of the proof.

Since $\boldsymbol{R}$ is characteristic in $E_{2}$, then $\Aut_{\H}\left(E_{2}\right)=\Aut_{\F_{\boldsymbol{2}}}\left(E_{2}\right)$.
Likewise $\Aut_{\H}\left(S\right)=\Aut_{\F_{\boldsymbol{1}}}\left(S\right)$.
From Proposition \ref{prop:centric-radicals-last-van-beek}, we conclude
that
\[
\F:=\left\langle \F_{\boldsymbol{1}},\F_{\boldsymbol{2}}\right\rangle _{S}=\H,
\]
and, therefore, $\F$ is saturated. Since both $\boldsymbol{R}$ and
$S$ are contained in $\mathcal{C}$, then $\Fcr_{x}\subseteq\mathcal{C}$
for every $x\in\Ob\left(\T\right)$ (see Lemma \ref{lem:centric-radical-contained}).
From Lemma \ref{lem:sharpness-for-normalizers}, each $\F_{x}$ satisfies
cohomological sharpness. 

Since $\boldsymbol{Q}$ is weakly $\F$-closed, then we also have
that $Q\trianglelefteq S$ and that
\[
\G:=\left\langle \G_{\boldsymbol{1}},\G_{\boldsymbol{2}}\right\rangle _{S}=N_{\F}\left(\boldsymbol{Q}\right).
\]

Since $\boldsymbol{Q}$ is maximal in $E_{2}$, then every morphism
in either $\G_{\boldsymbol{2}}$ or $\G_{\boldsymbol{e}}$ lifts to
a morphisms with source containing $E_{2}$ and normalizing it. Since
$\boldsymbol{R}$ is characteristic in $E_{2}$ it follows that $\G_{\boldsymbol{2}}=N_{\F}\left(E_{2}\right)$
and that $\G_{\boldsymbol{e}}=N_{\F_{1}}\left(E_{2}\right)$. It follows
from Corollary \ref{cor:whole-vanishing-of-CFR} (taken with $\H=\F_{1}$
and $P=Q$) that $\CFp[\Xi]\left(P\right)=0$ for every $P\in\mathcal{C}\backslash\left\{ \boldsymbol{Q}\right\} $.
From Corollary \ref{cor:relation-induction-no-induction}, we also
deduce that $\CFp[\Xi]\uparrow_{\OFC[\G]}^{\OFC}\left(\boldsymbol{Q}\right)\cong\CFp\left(\boldsymbol{Q}\right)$
via the morphism $\Gamma$ of Equation (\ref{eq:def-gamma}).

Since $\mathcal{E}\left(\F\right)=\left\{ E_{1}\right\} \cup E_{2}^{\F}$
and $E_{1}$ is characteristic $S$, then, for every $P\in\mathcal{C}$
satisying $P\not\le_{\F}E_{2}$, the identity $\Hom_{\F}\left(P,S\right)=\Hom_{\F_{\boldsymbol{1}}}\left(P,S\right)$
holds. From Lemma \ref{lem:vanishing-CFR}(\ref{enu:hom-included}),
we deduce that $\CFp\left(P\right)=0$ for every such $P$. Let $P\not\le_{\F}E_{1}$.
From Diagram (\ref{eq:hasse-diag-last-van-beek}), we obtain that
$\Hom_{\F}\left(P,S\right)=\Hom_{\G}\left(P,S\right)$ and that $\Hom_{\F_{x}}\left(P,S\right)=\Hom_{\G_{x}}\left(P,S\right)$
for every $x\in\Ob\left(\T\right)$. We conclude that $\CFp\left(P\right)=\CFp[\Xi]\uparrow_{\OFC[\G]}^{\OFC}\left(P\right)$
for every such $P$. Finally, a simple size comparison reveals that
the only subgroups of $E_{1}\cap E_{2}$ contained in $\mathcal{C}$
are $\boldsymbol{R}$ and $E_{1}\cap E_{2}$. In particular, $\boldsymbol{R}\le P$
for every $P\le E_{1}\cap E_{2}$ such that $P\in\mathcal{C}$. From
Diagram (\ref{eq:hasse-diag-last-van-beek}), we deduce that $\boldsymbol{R}$
is the only $\F$-conjugate of $\boldsymbol{R}$ contained in $P$.
In particular every $\F$-automorphism of $P$ leaves $\boldsymbol{R}$
invariant. we conclude that $\Aut_{\F}\left(P\right)=\Aut_{\F_{\boldsymbol{2}}}\left(P\right)$.
Let $P\not\cong_{\F_{\boldsymbol{e}}}P'\cong_{\F}P$. From Lemma \ref{lem:decomposition-aut-in-E1-last-van-beek.},
we deduce that $P'$ contains a conjugate of $\boldsymbol{R}$ different
than $\boldsymbol{R}$. From Diagram (\ref{eq:hasse-diag-last-van-beek}),
we conclude that $P'\boldsymbol{R}\not\le_{\F}E_{2}$. In particular
$\Hom_{\F_{\boldsymbol{2}}}\left(P',S\right)\subseteq\Hom_{\F_{\boldsymbol{1}}}\left(P',S\right)$.
It follows from Lemma \ref{lem:vanishing-CFR}(\ref{enu:aut-included})
that $\CFp\left(P\right)=\CFp[\Xi]\uparrow_{\OFC[\G]}^{\OFC}\left(P\right)=0$.

We conclude that the morphism $\Gamma$ of Equation \ref{eq:def-gamma}
is an isomorphism.

Finally, Proposition \ref{prop:sufficient-condition-for-splitting.}
allows us to apply Theorem \hyperref[thm:C]{C} with the introduced
notation. The result follows.
\end{proof}
\printbibliography

\end{document}